\documentclass[reqno,oneside,12pt]{amsart}

\usepackage[T1]{fontenc}
\usepackage{times,mathptm}
\usepackage{amssymb,epsfig,verbatim,xypic}
\theoremstyle{plain}

\newtheorem{thm}{Theorem}[section]

\newtheorem{cor}[thm]{Corollary}
\newtheorem{pro}[thm]{Proposition}
\newtheorem{lem}[thm]{Lemma}
\newtheorem{proposition-principale}[thm]{Proposition principale}
\newtheorem{thm-principal}{Th\'eor\`eme principal}[section]

\theoremstyle{definition}
\newtheorem{defi}[thm]{Definition}

\newtheorem{eg}[thm]{Example}
\newtheorem{rem}[thm]{Remark}

\newenvironment{thm-A}
{\noindent{\bf Theorem A.}\it}{\\}

\newenvironment{thm-AA}
{\noindent{\bf Theorem A'.}\it}{\\}

\newenvironment{thm-B}
{\noindent{\bf Theorem B.}\it}{\\}

\newenvironment{thm-BB}
{\noindent{\bf Theorem B'.}\it}

\def\C{\mathbf{C}}
\def\R{\mathbf{R}}
\def\Q{\mathbf{Q}}
\def\H{\mathbf{H}}
\def\Z{\mathbf{Z}}

\def\P{\mathbb{P}}

\def\Aut{{\sf{Aut}}}
\def\Bs{{\sf{Bs}}}

\def\PGL{{\sf{PGL}}\,}

\def\GL{{\sf{GL}}\,}
\def\SO{{\sf{SO}}\,}
\def\SU{{\sf{SU}}\,}
\def\SL{{\sf{SL}}\,}
\def\Sp{{\sf{Sp}}\,}
\def\End{{\sf{End}}\,}
\def\F{{\sf{F}}\,}
\def\Mat{{\sf{Mat}}\,}

\def\Her{{\sf{Her}}\,}

\def\SU{{\sf{SU}}\,}

\def\tr{{\sf{tr}}}

\def\rk{{\sf{rk}}}
\def\det{{\sf{det}}}

\def\aa{{\mathfrak{a}}}
\def\g{{\mathfrak{g}}}
\def\sll{{\mathfrak{sl}}}
\def\e{{\mathfrak{e}}}
\def\so{{\mathfrak{so}}}
\def\sp{{\mathfrak{sp}}}

\def\kod{{\text{kod}}}

\def\Pic{{\text{Pic}}}

\def\Ka{{\mathcal{K}}}
\def\Eff{{\mathcal{P}}s}
\def\PP{{\mathcal{P}}}
\def\Sym{{\text{Sym}}}

\def\dist{{\sf{dist}}}

\addtocounter{section}{0}             
\numberwithin{equation}{section}       

\begin{document}

\setlength{\baselineskip}{0.51cm}        
%
%
\title[Holomorphic actions on K\"ahler threefolds]
{Holomorphic actions of higher rank lattices in dimension three}
\date{2009}
\author{Serge Cantat and Abdelghani Zeghib}
\address{D\'epartement de math\'ematiques\\
         Universit\'e de Rennes\\
         Rennes\\
         France}
\email{serge.cantat@univ-rennes1.fr}
\address{CNRS \\
UMPA \\
\'Ecole
Normale Sup\'erieure de Lyon\\
France}
\email{zeghib@umpa.ens-lyon.fr}

%
%

%
%

%
%

\begin{abstract} 
We classify all holomorphic actions of higher rank lattices on
compact K\"ahler manifolds of dimension $3.$ This provides a complete
answer to Zimmer's program for holomorphic actions on compact
K\"ahler manifolds of dimension at most $3.$ \\

\noindent{\sc{R\'esum\'e.}} Nous classons les actions holomorphes
des r\'eseaux des groupes de Lie semi-simple de rang au moins $2$
sur les vari\'et\'es complexes compactes k\"ahl\'eriennes de dimension $3.$ 
Ceci r\'epond positivement au programme de Zimmer pour les
actions holomorphes en petite dimension.
\end{abstract}

\maketitle

\setcounter{tocdepth}{1}
\tableofcontents
%
%

\section{Introduction}
%
%

\subsection{Zimmer's program and automorphisms}$\,$

\vspace{0.16cm}

This article is inspired by two questions concerning the structure of 
groups of diffeomorphisms of compact manifolds. 

The first one is part of the so called Zimmer's program. Let $G$ be a  
semi-simple real Lie group. The {\bf{real 
rank}} $\rk_\R(G)$ of $G$ is the
dimension of a maximal abelian subgroup $A$ of $G$ such 
that ${\text{ad}}(A)$ acts by simultaneously
$\R$-diagonalizable endomorphisms on the Lie algebra $\g$ of $G.$   
Let us suppose that $\rk_\R(G)$ is at least $2$; in that case, we shall say
that $G$ is a {\bf{higher rank}} semi-simple Lie group.
Let $\Gamma$ be a lattice in $G$; by
definition, $\Gamma$ is a discrete subgroup of $G$ such that
$G/\Gamma$ has finite Haar volume. 
Margulis superrigidity theorem
implies that all finite dimensional linear representations of $\Gamma$ are built from 
representations in unitary groups and representations of the Lie group $G$
itself. Zimmer's program predicts that a similar picture should hold
for actions of $\Gamma$ by diffeomorphims on compact manifolds, 
at least when the dimension of the manifold is close to the minimal 
dimension of non trivial linear representations of $G$ (see \cite{Fisher:preprint}
for an introduction to Zimmer's program).   

The second problem that we have in mind concerns the structure of groups of 
holomorphic diffeomorphisms, also called {\bf{automorphisms}}.
Let $M$ be a compact complex manifold of dimension $n.$
According  to Bochner and Montgomery \cite{Bochner-Montgomery:1946, Campana-Peternell:survey}, the group of automorphisms  $\Aut(M)$
is a complex Lie group, the Lie algebra
of which is the algebra of holomorphic vector fields on $M.$ The connected 
component of the identity $\Aut(M)^0$ can be studied by classical means, namely 
Lie theory concerning the action of Lie groups on manifolds. 
The group of connected components 
$$
\Aut(M)^\sharp=\Aut(M)/\Aut(M)^0
$$
is much harder to describe, even for projective manifolds.

In this article, we provide a complete picture of all holomorphic actions of lattices 
$\Gamma$ in higher rank simple Lie groups on compact K\"ahler manifolds $M$
with $\dim(M)\leq 3.$ The most difficult part is the study of morphisms $\Gamma\to \Aut(M)$ 
for which the natural projection onto $\Aut(M)^\sharp$ is injective. As a consequence,
we hope that
our method will shed light on both Zimmer's program and the structure of $\Aut(M)^\sharp.$ 

\subsection{Examples: Tori and Kummer orbifolds}$\,$

\vspace{0.16cm}

Let us give classical examples coming from carefully chosen complex tori.

\begin{eg}\label{eg:tores}
Let $E=\C/\Lambda$ be an elliptic curve and $n$ be a positive integer. Let $A$ be
the torus  $E^n=\C^n/\Lambda^n.$ The group $\Aut(A)$
contains all affine transformations $z\mapsto B(z) + c$ where $B$ is in 
$\SL_{n}(\Z)$ and $c$ is in $A.$ The connected component $\Aut(A)^0$ 
coincides with the group of translations.
Similarly, if $\Lambda$ is the lattice of integers  ${\mathcal{O}}_{d}$
in an imaginary quadratic number field $\Q(\sqrt{d}),$ where $d$ is a
squarefree negative integer, then $\Aut(A)$ contains a
copy of $\SL_n({\mathcal{O}}_{d}).$ 
\end{eg}

\begin{eg}\label{eg:toreskummer}
Starting with the previous example, one can  
change $\Gamma$ in a finite index subgroup $\Gamma_0,$ and change $A$
into a quotient $A/G$ where $G$ is a finite subgroup of $\Aut(A)$ which 
is normalized by $\Gamma_0.$ In general, $A/G$ is an orbifold (a compact 
manifold with quotient singularities), and one needs to resolve the
singularities in order to get an action on a smooth manifold $M.$ 
The second operation that can be done is 
blowing up  finite orbits of $\Gamma.$ This provides infinitely many 
compact K\"ahler manifolds with actions of lattices $ \Gamma\subset\SL_n(\R)$ (resp. $\Gamma\subset\SL_n(\C)$). \end{eg}

In these examples, the group $\Gamma$ 
is a lattice in a real Lie group of rank $(n-1),$ namely $\SL_n(\R)$ or  $\SL_n(\C),$
and $\Gamma$
acts on a manifold $M$ of dimension~$n.$ Moreover, the action of $\Gamma$ on the
cohomology of $M$ has finite kernel and a finite index subgroup of $\Gamma$ embeds in $\Aut(M)^\sharp.$
Since this kind of construction is at the heart of the article, we introduce the
following definition.

\begin{defi}
{\sl{Let $\Gamma$ be a group, and $\rho:\Gamma\to \Aut(M)$ a morphism
into the group of automorphisms of a compact complex manifold $M.$ This
morphism is a {\bf{Kummer example}} (or, equivalently, is of {\bf{Kummer type}}) if there
exists
\begin{itemize}
\item a birational morphism $\pi:M\to M_0$ onto  an orbifold $M_0$, 
\item a finite cover $\epsilon:A\to M_0$ of $M_0$ by a torus $A,$ and
\item a morphism $\eta:\Gamma \to \Aut(A)$
\end{itemize} 
such that 
$
\epsilon \circ \eta(\gamma)
=
(\pi\circ \rho(\gamma)\circ \pi^{-1})\circ \epsilon  
$
for all $\gamma$ in $\Gamma.$}}
\end{defi}

The name Kummer comes from the fact that
 the orbifolds $A/G,$ $G$ a finite group, are known as
{\bf{Kummer orbifolds}}. Examples \ref{eg:tores} and \ref{eg:toreskummer}
are both Kummer examples.
 If $A=\C^n/\Lambda$ is a torus of dimension $n,$ every element of
$\Aut(A)$ is induced by an affine transformation of $\C^n.$ Hence, actions of Kummer type are covered by affine actions on $\C^{\dim(M)}.$

\subsection{Groups and lattices}$\,$

\vspace{0.16cm}

Before stating our results, a few classical definitions need to be given. 
Let $H$ be a group. A property is said to hold {\bf{virtually}} for $H$ if
a finite index subgroup of $H$ satisfies this property. For example, 
an action of a group $\Gamma$ on a compact complex manifold $M$
is virtually a Kummer example if this action is of Kummer type after restriction
to a finite index subgroup $\Gamma_0$ of $\Gamma.$ 

A connected  real Lie group $G$ is said to be {\bf{almost simple}} 
if its center is finite and its Lie algebra $\g$ is a simple Lie algebra. Let $G$ be a connected semi-simple real Lie group with finite center; then $G$ 
is isogenous to a product of almost simple Lie groups $G_i.$ The
groups $G_i$ are the {\bf{factors}} of $G.$
A lattice $\Gamma$ in $G$ is said to 
be {\bf{irreducible}} when its projection on every simple factor of $G$ is dense. 
A {\bf{higher rank lattice}} is a lattice in a connected semi-simple real Lie group 
$G$ with finite center and rank $\rk_\R(G)\geq 2.$

\subsection{Main results}$\,$

\vspace{0.16cm}

Let $\Gamma$ be an irreducible lattice in a higher rank almost simple Lie group.
As we shall see, all holomorphic faithful actions of $\Gamma$ on connected compact K\"ahler manifolds of dimension $1$ or $2$ are actions on the projective plane $\P^2(\C)$ by projective transformations.
As a consequence, we mostly consider actions on compact K\"ahler manifolds
of dimension $3.$ 

\vspace{0.2cm}

\begin{thm-A}
Let $G$ be a connected semi-simple real Lie group with finite center and without nontrivial compact factor. Let $\Gamma$ be an irreducible lattice in $G.$ Let $M$
be a connected compact K\"ahler manifold of dimension $3,$ and $\rho:\Gamma\to \Aut(M)$ be 
a morphism. If the real rank of $G$ is at least $2,$ then one of the following holds
\begin{itemize}
\item the image of $\rho$ is virtually contained in the connected component of
the identity $\Aut(M)^0,$ or
\item the morphism $\rho$ is virtually a Kummer example. 
\end{itemize}
In the second case, $G$ is locally isomorphic
to $\SL_3(\R)$ or $\SL_3(\C)$ and $\Gamma$ is commensurable to 
$\SL_3(\Z)$ or $\SL_3({\mathcal{O}}_{d}),$ where ${\mathcal{O}}_{d}$
is the ring of integers in an imaginary quadratic number field $\Q({\sqrt{d}})$ 
for some negative integer $d.$
\end{thm-A}

\vspace{-0.5cm}

\begin{rem}
In the Kummer case, the action of $\Gamma$ on $M$ comes virtually 
from a linear action of $\Gamma$ on a torus $A.$ We shall prove
 that  $A$ is isogenous to 
the product $B\times B \times B,$ where $B$ is an elliptic curve;
for example, one can take $B=\C/{\mathcal{O}}_{d}$ if $\Gamma$
is commensurable to  $\SL_3({\mathcal{O}}_{d}).$  \end{rem}

When the image of $\rho$ is virtually contained in $\Aut(M)^0,$
one gets a morphism from a sublattice $\Gamma_0\subset \Gamma$ 
to the connected complex Lie group $\Aut(M)^0.$ 
If the image is infinite, one can then find
a non trivial morphism of Lie groups from $G$ to $\Aut(M)^0,$ and use
Lie theory to describe all  types of manifolds $M$ that
can possibly arise. This leads to the following result.

\vspace{0.2cm}

\begin{thm-B}
Let $G$ be a connected, almost simple, real  Lie group with
rank $\rk_\R(G)\geq 2.$ Let $\Gamma$ be a lattice in $G.$
 Let $M$ be a connected compact K\"ahler
manifold of dimension $3.$ If there is a morphism $\rho:\Gamma\to \Aut(M)$ 
with infinite image,  then $M$ has a birational morphism onto
a Kummer orbifold, or $M$ is isomorphic to one of the following
\begin{enumerate}
\item a projective bundle $\P(E)$ for some rank $2$ vector bundle $E\to \P^2(\C),$ 
\item a principal torus bundle over $\P^2(\C),$ 
\item a product $\P^2(\C)\times B$ of the plane by a curve of genus $g(B)\geq 2,$ 
\item the projective space $\P^3(\C).$ 
\end{enumerate}
In all cases, $G$ is locally isomorphic to $\SL_n({\mathbf{K}})$ with $n=3$ or $4$ 
and ${\mathbf{K}}=\R$ or $\C.$
\end{thm-B} 

\vspace{-0.1cm}

One feature of our article is that we do not assume that the group
$\Gamma\subset \Aut(M)$ preserves a geometric structure or a volume
form. The existence of invariant structures comes as a byproduct of the
classification. For example, section \ref{par:VolumeForm} shows that any
holomorphic action of a lattice in a higher rank simple Lie group $G$ on 
a compact K\"ahler threefold either extends virtually to an action of 
$G$ or preserves a volume form (this volume form may have poles, but is
locally integrable).

Both theorem A and theorem B can be extended to 
arbitrary semi-simple Lie groups $G$ 
with $\rk_\R(G)\geq 2.$ Section \ref{par:GT} explains how to remove the 
extra hypothesis concerning the center and the compact factors of $G.$ 
The literature on lattices
in semi-simple Lie groups almost always assume that $G$
has finite center and no compact factor; this is the reason why 
we first focus on more restrictive statements.
 
\subsection{Organization of the paper}$\,$

\vspace{0.16cm}

Section
\ref{par:actionsofliegroups} explains how to deduce theorem  B from theorem A.
Let us now sketch the proof of theorem A and describe the organization 
of the paper. Let $\Gamma$ and $M$ be as in theorem A.

\vspace{0.16cm}

\noindent{ (1) } Section \ref{chap:prelimi} describes Lieberman-Fujiki results, Hodge index theorem, and
classical facts concerning lattices, including Margulis superrigidity theorem. Assuming that the
image of $\Gamma$ in $\Aut(M)$ is not virtually contained in $\Aut(M)^0,$
we deduce that  the action of $\Gamma$ on the
cohomology of $M$ extends virtually to a linear representation 
$$
(*)\quad G\to \GL(H^*(M,\R)),
$$
preserving the Hodge structure, the Poincar\'e duality, and the cup product.

\vspace{0.16cm}

\noindent{ (2) } Section \ref{chap:invariant} describes the geometry of possible $\Gamma$-invariant 
analytic subsets of $M.$ In particular, there is no $\Gamma$-invariant curve
and all $\Gamma$-periodic surfaces can be blown down simultaneously
by a birational morphism 
$$
\pi:M\to M_0
$$ 
onto an orbifold. This section
makes use of basic fundamental ideas from holomorphic dynamics and complex algebraic geometry.

\vspace{0.16cm}

\noindent{ (3) } Section \ref{chap:HSHRLG} classifies linear representations of $G$ on the cohomology of a compact K\"ahler threefold, as given by the 
representation $(*).$ It turns out that $G$ must be locally 
isomorphic to $\SL_3(\R)$ or $\SL_3(\C)$ when this representation is
not trivial. The proof uses highest 
weight theory for linear representations together with Hodge index theorem. 

\vspace{0.16cm}

\noindent{ (4) } In section \ref{chap:SL3I} and \ref{chap:SL3II}, we assume that $G$ is locally isomorphic to $\SL_3(\R).$
The representation of $G$ on the vector space
$$
W:=H^{1,1}(M,\R)
$$
and the cup product $\wedge$ from  $W\times W$ to its dual $W^*=H^{2,2}(M,\R)$
are described in section \ref{chap:SL3I}.
The representation $W$ splits as a direct sum of a trivial factor $T^m$ of dimension $m,$ $k$
copies of the standard representation of $G$ on $E=\R^3,$ and one copy of the
representation on the space of quadratic forms $\Sym_2(E^*).$
 We then study the position of $H^2(M,\Z)$ with respect
to this direct sum, and show that it intersects  $\Sym_2(E^*)$ on a cocompact lattice.
Section  \ref{chap:SL3I} ends with a study of the K\"ahler cone of $M.$
This cone is $\Gamma$-invariant, but is not $G$-invariant 
a priori.

\vspace{0.16cm}

\noindent{ (5) } Section \ref{chap:SL3II} shows that the trivial factor $T^m$ is spanned by 
classes of $\Gamma$-periodic surfaces. In particular, the map $\pi:M\to M_0$ 
contracts $T^m.$ To prove this, one pursues the study of the K\"ahler, nef, and big cones.
This requires a characterization of K\"ahler classes due to Demailly and Paun.

\vspace{0.16cm}

\noindent{ (6) } The conclusion follows from the following observation: On $M_0,$
both Chern classes $c_1(M)$ and $c_2(M)$ vanish, and this ensures that
$M_0$ is a Kummer orbifold. One then prove that $\Gamma$ 
is commensurable to $\SL_3(\Z).$

\vspace{0.16cm}

\noindent{ (7) } Section \ref{chap:SL3C} concerns the case when $G$ is locally isomorphic 
to $\SL_3(\C).$ We just give the main lemmas that are needed to adapt
the proof given for $\SL_3(\R).$ All together, this concludes the proof of theorem A.

\subsection{Notes and references}$\,$

\vspace{0.16cm}

This article is independant from \cite{Cantat:ENS}, but may be considered 
as a companion to  this article in which the first author proved 
that a lattice in a simple Lie 
group $G$ can not act faithfully by automorphisms on a compact K\"ahler 
manifold $M$ if $\rk_\R(G)> \dim_\C(M).$ 

Several arguments are inspired by previous works concerning
actions of lattices by real analytic transformations (see \cite{Ghys:1993}, \cite{Farb-Shalen:1999, Farb-Shalen:2002}), or morphisms from lattices to mapping class groups (see \cite{Farb-Shalen:2000, Farb-Masur:1998}). 

Concerning actions of lattices on tori, the interesting reader may consult 
\cite{Katok-Lewis:1991, Katok-Lewis:1996, Benveniste-Fisher:2005}. 
These references contain interesting examples that can not appear in the holomorphic setting. 
A nice introduction to Zimmer's program is given in \cite{Fisher:preprint}.

\subsection{Aknowledgments}$\,$

\vspace{0.16cm}

Thanks to Bachir Bekka, David Fisher, \'Etienne Ghys, Antonin Guilloux, Yves Guivarch, Fran\c{c}ois Maucourant, and
Dave Morris for useful discussions.
 
%
%

\section{Automorphims, rigidity, and Hodge theory}\label{chap:prelimi}

%
%

In this section, 
we collect important results from Hodge theory and the theory of
discrete subgroups of Lie groups, which will be used systematically
in the next sections. 

In what follows, $M$ is a compact K\"ahler manifold. Unless
otherwise specified, $M$ is assumed to be connected.

\subsection{Automorphims}\label{par:Lieberman-Fujiki}$\,$

\vspace{0.16cm}

As explained in the introduction, $\Aut(M)$ is a complex
Lie group, but the group $\Aut(M)^\sharp$ of its connected components can be infinite, as in example \ref{eg:tores}. The following theorem shows that $\Aut(M)^\sharp$ embeds virtually into 
the linear group $\GL(H^2(M,\R)),$ and thus provides a way to study it.

\begin{thm}[Lieberman, Fujiki, see \cite{Lieberman:1978}]\label{thm:LF} Let $M$ be a compact K\"ahler
manifold. Let $\kappa$ be the cohomology class of
a K\"ahler form on $M.$  The connected component of the identity  $\Aut(M)^0$ 
has finite index in the stabilizer of $\kappa$ in~$\Aut(M).$ 
\end{thm}

\subsection{Hodge structure}

\subsubsection{Hodge decomposition} Hodge theory implies that 
the cohomology groups $H^k(M,\C)$ decompose into direct sums
$$
H^k(M,\C)= \bigoplus_{p+q=k} H^{p,q}(M,\C), 
$$
where cohomology classes in $H^{p,q}(M,\C)$ are represented by closed
forms of type $(p,q).$ This bigraded structure is compatible with the cup product. 
Complex conjugation permutes $H^{p,q}(M,\C)$
with $H^{q,p}(M,\C).$ In particular, the cohomology groups $H^{p,p}(M,\C)$ 
admit a real structure, the real part of which is $H^{p,p}(M,\R)= H^{p,p}(M,\C)\cap H^{2p}(M,\R).$ 
If $\kappa$ is a K\"ahler class (i.e. the cohomology class of a K\"ahler form),
then $\kappa^p \in H^{p,p}(M,\R)$ for all $p.$

\subsubsection{Dimension three}\label{par:ActiononH*}
Let us now assume that $M$ has dimension $3.$ 
In what follows, we shall denote by $W$ the cohomology group $H^{1,1}(M,\R).$ 
Poincar\'e duality 
provides an isomorphism between $H^{2,2}(M,\R)$ and $W^*,$
which is defined by 
$$
\langle \alpha \vert \beta \rangle = \int_M \alpha \wedge \beta
$$
for all $\alpha \in H^{2,2}(M,\R)$ and $\beta \in W.$
Modulo this isomorphism, the cup product defines a symmetric bilinear 
application 
$$
\wedge:  \left\{ 
\begin{array}{ccc}
W \times  W & \to & W^* \\
(\beta_1,\beta_2) & \mapsto & \beta_1\wedge \beta_2
\end{array}
\right.
$$

\begin{lem}[Hodge Index Theorem]\label{lem:hit}
If $\wedge$ vanishes identically along a subspace $V$ of $W,$
the dimension of $V$ is at most $1.$ 
\end{lem}

\begin{proof}
Let $V$ be a linear subspace of $W$ along which $\wedge$ vanishes identically: 
$\alpha \wedge \beta = 0,$ for all $\alpha$ and $\beta$ in $V.$
Let $\kappa\in W$ be a K\"ahler class. Let $Q_\kappa$ be the quadratic 
form which is defined on $W$ by
$$
Q_\kappa(\alpha_1,\alpha_2)= -\int_M \alpha_1\wedge \alpha_2\wedge \kappa,
$$
and $P^{1,1}(\kappa)$ be the orthogonal complement of $\kappa$ with 
respect to this quadratic form: 
$$
P^{1,1}(\kappa)= \left\{ \beta \in W \, ; \quad Q_\kappa(\beta,\kappa) = 0 \right\}.
$$
Since $\kappa$ is a K\"ahler class, $\kappa\wedge\kappa$ is different from $0$  and
$P^{1,1}(\kappa)$ has codimension $1$ in $W.$ 
Hodge index theorem implies that $Q_\kappa$ is positive definite on $P^{1,1}(\kappa)$ (see \cite{Voisin:book}, theorem 6.32).
In particular, $P^{1,1}(\kappa)$ does not intersect $V,$ and the dimension of $V$
is at most $1.$
\end{proof}

\subsection{Classical results concerning lattices}$\,$

\vspace{0.16cm}

Let us list a few important facts concerning lattices in Lie groups. 
The reader may consult \cite{VGS:EMS} and \cite{Benoist:cours}
for two nice introductions to lattices. One feature of the theory
of semi-simple Lie groups is that we can switch  viewpoint
from Lie groups to linear algebraic groups. We shall use this
almost permanently.

\subsubsection{Borel and Harish-Chandra (see \cite{Benoist:cours,VGS:EMS})}\label{par:BHC} Let $G$ 
be a linear algebraic group defined over the field of rational 
numbers $\Q.$ Let $ {\mathbf{G}}_m$ denote the multiplicative
group. If $G$ does not have any character $G\to {\mathbf{G}}_m$ 
defined over $\Q,$ the group of integer points $G(\Z)$ is a lattice
in $G.$ If there is no morphism ${\mathbf{G}}_m\to G$ defined over
$\Q,$ this lattice is cocompact. 

\subsubsection{Borel density theorem (see \cite{VGS:EMS}, page 37, or \cite{Benoist:cours})}\label{par:boreldensity} Let $G$ be a linear algebraic semi-simple Lie group with no compact normal subgroup of positive dimension. 
If $\Gamma$ is a lattice in $G,$ then $\Gamma$ is Zariski-dense in $G.$ 

\subsubsection{Proximality (see \cite{Benoist-Labourie:1993}, appendix)}\label{par:proximality} 
Let $G$ be a real reductive linear algebraic group
and $P$ be a minimal parabolic subgroup of $G.$ 
An element $g$ in $G$ is {\bf{proximal}} if it has an attractive fixed point in 
$G/P.$ For example, when $G=\SL_n(\R),$ an element is proximal if and only if its 
eigenvalues are $n$ real numbers with pairwise distinct absolute values. 

Let $\Gamma$ be a Zariski dense subgroup of $G.$ Then, the set of proximal 
elements $\gamma$ in $\Gamma$ is Zariski dense in $G.$ More precisely, 
when $G$ is not commutative, 
there exists a Zariski dense non abelian free subgroup  $F< \Gamma$ such that all elements $\gamma$ in $F\setminus \{{\text{Id}} \}$ are proximal.

\subsubsection{Limit sets (see \cite{Mostow:book}, lemma 8.5)}\label{par:limitset} Let $G$ be a semi-simple analytic group having no compact 
normal subgroup of positive dimension. Let $P$ be a parabolic subgroup of
$G.$ If $\Gamma$ is a lattice in $G,$ the closure of $\Gamma P$ coincides 
with $G$: $\overline{\Gamma P}= G.$ In particular, if $\Gamma$ is a lattice 
in $\SL_3({\mathbf{K}}),$ ${\mathbf{K}}=\R$ or $\C,$  all
orbits of $\Gamma$ in $\P^2({\mathbf{K}})$ are dense.

\subsubsection{Kazhdan property $(T)$ (see \cite{delaHarpe-Valette,BHV})}\label{par:PropT}
We shall say that a topological  group $F$ has Kazhdan property  (T) if $F$ is locally 
compact and every continuous action of $F$ by affine unitary motions on a Hilbert space has a fixed point (see \cite{BHV} for equivalent definitions). 
If $F$ has Kazhdan property (T) and $\Lambda$ is a lattice in $F,$ then $\Lambda$
inherits property (T). 
If $G$ is a simple real Lie group with rank $\rk_\R(G)\geq 2,$ then $G$ and all
its lattices satisfy Kazhdan property $(T)$; moreover, the same property holds
for the universal covering of $G.$ 

If $F$ is a discrete group with property (T), then 
\begin{itemize}
\item  $F$ is finitely generated; 
\item every morphism from $F$ to $\GL_2(k),$ $k$ any field, has finite image
(see \cite{Guentner-all:2005} and \cite{Zimmer:1984});
\item every morphism from $\Gamma$  to a solvable group has finite image.
\end{itemize}

\begin{lem}\label{lem:TtoCurve}
Every morphism from a discrete Kazhdan group $F$ to the group
of automorphisms of a compact Riemann surface has finite image.
\end{lem}

\begin{proof}
The automorphisms group of a connected Riemann surface $X$
is either finite (when the genus $g(X)$ is $>1$), virtually abelian (when $g(X)=1$), or isomorphic 
to $\PGL_2(\C).$ \end{proof}

\subsection{Margulis superrigidity and action on cohomology}

\subsubsection{Superrigidity} 
The following theorem is one version of the superrigidity phenomenum for linear representations of lattices (see \cite{Margulis:Book} or \cite{VGS:EMS}). 

\begin{thm}[Margulis]\label{thm:SuperMargu}
Let $G$ be a semi-simple connected Lie group with finite center, with rank at least $2,$ and without non trivial compact factor.
Let $\Gamma\subset G$ be an irreducible lattice. 
Let $h:\Gamma \to \GL_k(\R)$ be a linear representation of $\Gamma.$

The Zariski closure of $h(\Gamma)$ is a semi-simple Lie group; if this
Lie group does not have any infinite compact factor, there exists a continuous representation 
${\hat{h}}:G\to \GL_k(\R)$ which coincides with $h$ on a finite index subgroup
of $\Gamma.$
\end{thm}

In other words, if the Zariski closure of $h(\Gamma)$ does not have any non trivial compact factor, then $h$ virtually extends to a linear representation of $G$. 

\begin{rem}
Similar results hold also for representations of lattices in $\Sp(1,n)$ and $\F_4$ 
(see \cite{Corlette:1992}).
\end{rem}

\begin{cor}\label{cor:margulis}
If the representation $h$ takes values
into the group $\GL_k(\Z)$ and has an infinite image, 
then $h$ extends virtually to a continuous representation of $G$ with finite kernel.
\end{cor}

If $G$ is a linear algebraic group, any continuous linear representation of $G$
on a finite dimensional vector space is algebraic. 
As a consequence, up to finite indices, representations
of $\Gamma$ into $\GL_k(\Z)$ with infinite image are restrictions of algebraic
linear representations of $G.$ 

\subsubsection{Normal subgroups (see \cite{Margulis:Book} or \cite{VGS:EMS})}\label{par:normalsubgroups}

According to another result of Margulis, if $\Gamma$ is an irreducible higher rank lattice, 
then $\Gamma$ is almost simple: All normal subgroups of
$\Gamma$ are finite or cofinite (i.e. have finite index in $\Gamma$). 

In particular, if $\alpha:\Gamma\to H$ is a morphism of groups, either $\alpha$
has finite image, or $\alpha$ is virtually injective, which means that 
we can change the lattice $\Gamma$ in a sublattice $\Gamma_0$ and
assume that $\alpha$ is injective.

\subsubsection{Action on cohomology groups} 

From Margulis superrigidity and Lie\-ber\-man-Fujiki theorem, one gets the following proposition.

\begin{pro}\label{pro:ExtensionCohomology}
Let  $G$ and $\Gamma$ be as in theorem \ref{thm:SuperMargu}.
Let $\rho:\Gamma \to \Aut(M)$ be a representation into the
group of automorphisms
of a compact K\"ahler manifold $M.$ Let $\rho^* :\Gamma\to \GL(H^*(M,\Z))$ be the
induced action on the cohomology ring of $M.$ 
\begin{itemize}
\item[a.-] If the image of $\rho^*$ is
infinite, then $\rho^*$ virtually extends to a representation 
${\hat{\rho^*}}:G\to \GL(H^*(M,\R))$ which preserves the cup product and
the Hodge decomposition.
\item[b.-] If the image of $\rho^*$ is finite, the image of $\rho$ is virtually contained
in $\Aut(M)^0.$ 
\end{itemize}
\end{pro}

Hence, to prove theorem A, we can assume that the action of the lattice
$\Gamma$ on the cohomology of $M$ extends to a linear representation of $G.$

\begin{proof}
In the first case, Margulis superrigidity implies that the morphism $\rho^*$ extends to 
a linear representation ${\hat{\rho^*}}$ of $G.$  
Since $\Gamma$ acts by holomorphic diffeomorphisms on $M,$
$\Gamma$ preserves the Hodge decomposition and the cup product.
Since lattices of semi-simple Lie groups are Zariski dense (see \S \ref{par:boreldensity}), the same is true for  ${\hat{\rho^*}}(G).$ 

The second assertion is a direct consequence of theorem \ref{thm:LF}\end{proof}

For compact K\"ahler threefolds, section 
\ref{par:ActiononH*} shows that the Poincar\'e duality and the cup product define a
bilinear map $\wedge$ on $W=H^{1,1}(M,\R)$ with values into the dual
space $W^*.$ 

\begin{lem} \label{lem:ExtensionCohomology}
In case (a) of proposition \ref{pro:ExtensionCohomology}, 
the bilinear map $\wedge: W \times W \to W^*$ is $G$-equivariant: 
$$
({\hat{\rho^*}}(g)\alpha)\wedge ({\hat{\rho^*}}(g)\beta)= ({\hat{\rho^*}}(g))^* (\alpha\wedge \beta)
$$ 
for all $\alpha,$ $\beta$ in $W.$ 
\end{lem}

\subsubsection{Notations} In the proof of theorem A, we shall simplify the notation,
and stop refering explicitly to $\rho,$ $\rho^*,$ 
$\hat \rho^*.$ If the action of $\Gamma$ on $W$ extends virtually to an action 
of $G,$ this action is denoted  by $(g,v)\mapsto g(v).$ This representation preserves
the Hodge structure, the Poincar\'e duality and the cup product.  

\subsection{Dynamical degrees and Hodge decomposition}

\subsubsection{Dynamical degrees (see \cite{Guedj:ETDS, Guedj:survey} and \cite{Dinh-Nguyen:2009})}\label{par:dynadegrees}

Let $f$ be an automorphism of a compact K\"ahler manifold $M.$ 
Let $1\leq p \leq \dim(M)$ be a positive integer. The dynamical
degree of $f$ in dimension $p$ is the spectral radius of 
$$
f^*: H^{p,p}(M,\R)\to H^{p,p}(M,\R).
$$
We shall denote it by $d_p(f).$ One easily shows that $d_p(f)$
coincides with the largest eigenvalue of $f^*$ on $H^{p,p}(M,\R).$
Hodge theory also implies that the following properties are
equivalent:
\begin{itemize}
\item[(i)] one of the dynamical degrees $d_p(f),$ $1\leq p \leq \dim(M)-1,$ is
equal to $1$;
\item[(ii)] all dynamical degrees $d_p(f)$ are equal to $1$;
\item[(iii)] the spectral radius of 
$f^*:H^*(M,\C)\to H^*(M,\C)$ is equal to $1.$
\end{itemize}

From theorems due to Gromov and Yomdin follows that the
topological entropy of $f:M\to M$ is equal to the maximum
of the logarithms $\log(d_p(f))$ (see \cite{Gromov:Enseignement}).

\subsubsection{Invariant fibrations}

Let us now assume that $f:M\to M$ is an automorphism
of a compact K\"ahler manifold $M,$ that $\pi:M\dasharrow B$ is a meromorphic
fibration, and that $f$ permutes the fibers of $\pi,$ by which 
we mean that there is a bimeromorphic map $g:B\dasharrow B$
such that $\pi\circ f=g\circ \pi.$ In this setting, one can define 
dynamical degrees $d_p(f\,\vert\, \pi)$ along the fibers of $\pi$
and dynamical degrees for $g$ on $B,$ and relate all these degrees
to those of $f$ (see \cite{Dinh-Nguyen:2009}).

\begin{thm}[Dinh and Nguyen]\label{thm:Dinh-VietAnh}
Under the previous assumptions, the dynamical degrees satisfy
$
d_p(f)=\max_{i+j=p} ( d_i(g) d_{j}(f\,\vert\, \pi)) 
$
for all $0\leq p\leq \dim(M).$ 
\end{thm}

\subsubsection{Dynamical degrees for actions of lattices}

Let us now consider an action of an irreducible higher rank lattice $\Gamma$
on a compact K\"ahler manifold $M,$ and assume that the
action of $\Gamma$ on the cohomology of $M$ does not 
factor through a finite group. Then, according to proposition
 \ref{pro:ExtensionCohomology},
the action of $\Gamma$ on $H^{1,1}(M,\R)$ extends virtually
to a non trivial linear representation of the Lie group $G.$ 
In particular, there are elements $\gamma$ in $\Gamma$
such that the spectral radius of 
$$
\gamma^*:H^{1,1}(M,\R)\to H^{1,1}(M,\R)
$$
is larger than $1$ (see \cite{Cantat:ENS}, \S 3.1 for more precise
results). This remark and the previous theorem of Dinh and
Nguyen provide obstructions for the existence of meromorphic fibrations
$\pi:M\dasharrow B$ that can be $\Gamma$-invariant. We
shall use this fact in the next section.

%
%

\section{Invariant analytic subsets}\label{chap:invariant}

%
%

This section classifies possible $\Gamma$-invariant
analytic subspaces $Z\subset M,$ where $\Gamma$ is a discrete Kazhdan group, 
or a higher rank lattice in a simple Lie group, acting faithfully on a (connected) compact K\"ahler threefold $M.$ 




\subsection{Fixed points and invariant curves}

\begin{thm} \label{thm:fixedpoints}
Let $\Gamma$ be a discrete group with Kazhdan 
property $(T),$ and $\Gamma\to \Aut(M)$ be a morphism into 
the group of automorphisms of a connected compact K\"ahler threefold $M.$ If the fixed points set  of 
$\Gamma$ in $M$ is infinite, the image of $\Gamma$ in $\Aut(M)$
is finite. \end{thm}

\begin{proof}
The set of fixed points of $\Gamma$ is the set
$$
\{m\in M \, \vert \, \gamma(m)=m, \, \, \forall \gamma\in \Gamma\}.
$$
Since $\Gamma$ acts holomorphically, this set is an analytic subset of $M$;
 we have to show  that its dimension is $0.$ 
Let  $Z$ be an irreducible component of this set with
positive dimension, and  $p$ be a smooth point of $Z.$ 
Since $p$ is
a fixed point, we get a linear representation $D:\Gamma\to \GL(T_pM)$ 
defined by the differential at $p$
$$
D: \gamma\mapsto   D_p  \gamma.
$$
 If $Z$ has dimension $1,$ then $D$
takes its values in the pointwise stabilizer of a line.
This group is isomorphic to the semi-direct product $\SL_2(\C)\ltimes \C^2.$ 
If $Z$ has dimension $2,$ $D$ takes its values in the pointwise stabilizer of a plane.
This group is solvable.  Since $\Gamma$ has property (T), the image of $D$ must be finite (see \S \ref{par:PropT}),
so that all elements in a finite index subgroup $\Gamma_0$ of $\Gamma$ are tangent
to the identity ar $p.$ For every $k>1,$ the group of $k$-jets of diffeomorphisms which are
tangent to the identity at order $k-1$ is an abelian group, isomorphic to the sum of
three copies of the group of homogenous polynomials of degree $k$ (in $3$ variables). 
Since $\Gamma_0$ has property $(T),$ there is no non trivial morphism from $\Gamma_0$ to such
an abelian group, and the Taylor expansion of every element of $\Gamma_0$ at $p$ is trivial. 
Since the action of $\Gamma_0$ is holomorphic  and $M$ 
is connected,  $\Gamma_0$ acts
trivially on $M,$ and the morphism $\Gamma\to \Aut(M)$ has finite image.  \end{proof}

\begin{cor}\label{cor:invcurves}
 If $\Gamma$ is an infinite discrete Kazhdan group
acting faithfully on a connected compact K\"ahler threefold $M,$  there is no $\Gamma$-invariant curve in~$M.$ 
\end{cor}

\begin{proof}
Let $C$ be an invariant curve. Replacing $\Gamma$ by a finite index
subgroup, we can assume that $\Gamma$ preserves all connected components
$C_i$ of $C.$ The conclusion follows from the previous theorem and
lemma \ref{lem:TtoCurve}. 
\end{proof}

\subsection{Invariant surfaces}\label{par:InvSurf}$\,$

\begin{thm}[see \cite{Cantat:ENS,Cantat:Cremona}]\label{thm:invsurfaces} Any holomorphic action with infinite
image of a  discrete Kazhdan group
 $\Gamma$ on a compact K\"ahler surface $S$ is holomorphically conjugate to an action on $\P^2(\C)$ by projective transformations. \end{thm}

We just sketch the proof of this result. Complete details are given in \cite{Cantat:Cremona} in a more general
context, namely actions by birational transformations (see also \cite{Cantat:ENS}
for hyperk\"ahler manifolds).

\begin{proof}[Sketch of the proof]
Let $\Gamma$ be a Kazhdan group, acting on a compact K\"ahler surface $S.$ 
The action of $\Gamma$ on $H^{1,1}(M,\R)$ preserves the intersection form. 
Hodge index theorem asserts that this quadratic form has signature $(1,\rho -1),$ where $\rho$ is the dimension of $H^{1,1}(M,\R).$ Let $\Ka(S)\subset H^{1,1}(S,\R)$ be the K\"ahler cone, i.e.
the open convex cone of cohomology classes of K\"ahler forms of $S.$ This cone
is $\Gamma$-invariant and is contained in the cone of cohomology classes
with positive self intersection. The fixed point property for actions of Kazhdan groups
on hyperbolic spaces implies that there exists a K\"ahler class $[\kappa]$ which
is $\Gamma$-invariant (see lemma 2.2.7 page 78 in \cite{BHV}). Lieberman-Fujiki theorem now implies
that the image of a finite index subgroup $\Gamma_0$  of $\Gamma$ is contained
in the connected Lie group $\Aut(S)^0.$ In particular, all curves with negative
self intersection on $S$ are $\Gamma_0$-invariant, so that we can blow down 
a finite number of curves in $S$ and get a birational morphism $\pi:S\to S_0$ onto 
a minimal model of $S$ wich is $\Gamma_0$-equivariant: $\pi\circ \gamma = \Psi(\gamma)\circ \pi$
where $\Psi:\Gamma_0\to \Aut(S_0)^0$ is a  morphism.

From lemma \ref{lem:TtoCurve}, we know that $\Psi(\Gamma_0)$ is finite as soon 
 as it  preserves a  fibration
$S\to B,$ with $\dim(B)=1.$ 
Let us now use this fact together with a little bit of the classification of compact complex surfaces
(see \cite{BPV:book}). 

If the Kodaira
dimension $\kod(S_0)$ is equal to $2,$ then $\Aut(S_0)$ is finite. We can therefore assume that
$\kod(S_0)\in \{1,0,-\infty\}.$ If $\kod(S_0)=1,$ then $S_0$ fibers over a curve in an $\Aut(S_0)$-equivariant way (thanks to the Kodaira-Iitaka fibration), and $\Psi(\Gamma_0)$ must be finite. 
If $\kod(S_0)=0,$ then $\Aut(S_0)^0$ is  abelian, so that the image of $\Gamma_0$ 
in $\Aut(S_0)$ is also finite (see \S \ref{par:PropT}). Let us now assume
that $\kod(S_0)=-\infty.$ If $S_0$ is a ruled surface over a
curve of genus $g(B)\geq 1,$ the ruling is an invariant fibration over the base $B.$
If $S_0$ is  an Hirzebruch surface, $\Aut(S_0)$ also preserves a non trivial
fibration in this case. In both cases, $\Psi(\Gamma_0)$ is finite. 
The unique remaining case is the projective plane. In particular, if the image of $\Gamma$ in 
$\Aut(S)$ is infinite, then $S_0$ is isomorphic to $\P^2(\C).$

Let us now assume that $S_0=\P^2(\C),$ so that $\Psi(\Gamma_0)$
is a Kazhdan subgroup of $\PGL_3(\C).$ The set of critical values of $\pi$ is a $\Gamma_0$-invariant
subset of $S_0.$ If this set is not empty, $\Psi(\Gamma_0)$ is contained in a strict,
Zariski closed subgroup of $\PGL_3(\C).$  All morphisms from a discrete Kazhdan group to 
such a subgroup of $\PGL_3(\C)$ have finite image (see \S \ref{par:PropT}). As a consequence, 
the critical locus of $\pi$ is empty, and $\pi$ is an isomorphism from $S$ to $\P^2(\C).$
\end{proof}

\begin{cor}\label{cor:allp2}
Let $M$ be a connected compact K\"ahler threefold, and $\Gamma\subset\Aut(M)$ be a discrete Kazhdan group.
Let $S\subset M$ be a $\Gamma$-invariant surface. Then $S$ is smooth and all its connected
components are isomorphic to $\P^2(\C).$ 
\end{cor}

\begin{proof}
Let us assume that the singular locus of $S$ is not empty.  
Then, this set must have dimension $0$ (corollary \ref{cor:invcurves}). Let $p$ 
be such a singular point and $\Gamma_0$ be the finite index
subgroup of $\Gamma$ fixing $p.$ The differential 
$$
 \gamma\mapsto D_p (\gamma)
$$
 provides a morphism $
D: \Gamma_0\to \GL(T_pM)
$ which preserves the tangent
cone to $S$ at $p.$   This cone 
defines a curve in $\P(T_pM)$ which is invariant under the linear
action of $\Gamma_0.$ Lemma \ref{lem:TtoCurve} shows that  $\Gamma_0$ acts
trivially on this curve, and therefore on the cone. As a consequence, 
the morphism $D$ is trivial. The proof of theorem \ref{thm:fixedpoints} now provides
a contradiction. It follows that $S$ is smooth. 
Theorems \ref{thm:invsurfaces} and \ref{thm:fixedpoints} imply that the connected components of $S$ are
 isomorphic to $\P^2(\C).$ 
\end{proof}

\begin{thm}\label{thm:class-inv-div}
Let $M$ be a compact K\"ahler manifold of dimension $3.$
Assume that $\Aut(M)$ contains a subgroup $\Gamma$ such that
\begin{itemize}
\item[(i)] $\Gamma$ is an infinite discrete Kazhdan group;
\item[(ii)] there is an element $\gamma$ in $\Gamma$ with 
a dynamical degree $d_p(\gamma) > 1.$
\end{itemize}
If $Z$ is a $\Gamma$-invariant analytic subset of $M,$ it is made
of a finite union of isolated points and a finite union of disjoint smooth surfaces 
$S_i\subset M$ that are all isomorphic to the projective plane. Moreover, 
all of them are contractible to  quotient singularities. \end{thm}

\begin{rem} 
More precisely, we shall prove that each component $S_i$ can be blown down 
to a singularity which is locally isomorphic to the quotient of $\C^3$ by scalar
multiplication by a root of unity. 
\end{rem}  

\begin{rem}
We shall apply this result and its corollaries for lattices in higher rank simple Lie groups in section 
\ref{chap:SL3II}.
\end{rem}

\begin{proof}
From corollary \ref{cor:invcurves} we know that $Z$ is made of surfaces
and isolated points, and corollary \ref{cor:allp2} shows that all connected, $2$-dimensional
components of $Z$ are smooth projective planes. Let $S$ be one of them. 
Let $N_S$ be its normal bundle, and let $r$ be the integer such that
$N_S \simeq {\mathcal{O}}(r).$  Let $L$ be the line bundle on $M$ 
which is defined by $S.$
The adjunction formula shows that $N_S$ is the restriction $L_{\vert S}$ of $L$
to $S.$ 

If $r> 0,$ then $S$ moves in a linear system $\vert S \vert$ of positive dimension (see 
\cite{Hartshorne:LNM}); equivalently, the space of sections 
$H^0(M,L)$ has dimension $\geq 2.$  Moreover the 
line bundle $L_{\vert S}\simeq {\mathcal{O}}(r)$ is very ample, so that the base
locus of the linear system $\vert S\vert$ is empty. As a consequence, this linear system determines a well defined morphism
$$
\Phi_L:M\to \P(H^0(M,L)^*)
$$
where $\Phi_L(x)$ is the linear form which maps a section $s$ of $L$ to its value
at $x$ (this is well defined up to a choice of a coordinate along the fiber $L_x,$ i.e. up to a scalar multiple).
The self intersection $L^3$ being positive, the dimension of $\Phi_L(M)$ is equal to $3$
and $\Phi_L$ is generically finite. 
Since $S$ is $\Gamma$-invariant, $\Gamma$ permutes linearly the sections of $L.$ 
This gives a morphism $\eta : \Gamma\to \PGL(H^0(M,L)^*)$ such that $\Phi_L\circ \gamma = \eta(\gamma)\circ \Phi_L$ for all $\gamma$ in~$\Gamma.$ 
Let $\gamma$ be an element of $\Gamma.$
The action of $\eta(\gamma)$ on $\Phi_L(M)$ is induced by a linear mapping. 
This implies that $(\eta(\gamma))^*$ has finite order on the cohomology groups of
$\Phi_L(M).$ Since $\Phi_L$ is generically finite to one, all eigenvalues of
$\gamma$ on $H^{1,1}(M,\R)$ have modulus $1$ (theorem \ref{thm:Dinh-VietAnh}), contradicting  assumption (ii).

Let us now assume that $r=0.$ From \cite{Kodaira-Spencer:1959}, we know that 
$S$ moves in a pencil of surfaces. In other words, $H^0(M,L)$ has dimension
$2$ and defines a holomorphic fibration
$
\Phi_L:M\to B
$
where $B$ is a curve. This fibration is $\Gamma$-invariant, and replacing $\Gamma$ 
by a finite index subgroup, we can assume that the action on the base $B$ is trivial
(lemma \ref{lem:TtoCurve}). All fibers of $\Phi_L$ are $\Gamma$-invariant surfaces and, 
as such, are smooth projective planes. This implies that $\Phi_L$ is a locally trivial fibration: There is a covering $U_i$
of $B$ such that $\Phi_L^{-1}(U_i)$ is isomorphic to $U_i\times \P^2(\C).$ In such 
coordinates, $\Gamma$ acts by 
$$
\gamma(u,v)=(u, A_u(\gamma)(v))
$$
where $(u,v)\in U_i\times \P^2(\C)$ and $u\mapsto A_u$ is a one parameter
representation of $\Gamma$ into $\Aut(\P^2(\C)).$ Once again, theorem
\ref{thm:Dinh-VietAnh} 
implies that all eigenvalues of $\gamma$ on $H^{1,1}(M,\R)$ have modulus $1$ (see \cite{Dinh-Nguyen:2009}),
a contradiction.

As a consequence, the normal bundle $N_S$ is isomorphic to ${\mathcal{O}}(r)$
with $r<0.$ Grauert's theorem shows that $S$ can be blown down to a quotient singularity 
of type $(\C^3,0)/\xi$ where $\xi$ is a root of unity of order $r.$ 
\end{proof}

\begin{cor}\label{cor:no-inv-fib}
If $M$ and $\Gamma$ are as in theorem \ref{thm:class-inv-div},
there is no $\Gamma$-equivariant fibration $p:M\to B$ with $\dim(B)\in \{1,2\}.$ 
\end{cor}

\begin{proof}
If there is such a fibration with $\dim(B)=1,$ then the action of $\Gamma$ on
the base $B$ is virtually trivial, and the fibers are virtually invariant. We then get
a contradiction with  the previous theorem. If the dimension of $B$ is $2,$ then 
two cases may occur. The action of $\Gamma$ on the base is virtually trivial, 
and we get a similar contradiction, or $B$ is isomorphic to the plane and
$\Gamma$ acts by projective transformations on it. 
Once again, the contradiction comes from theorem \ref{thm:Dinh-VietAnh}.
\end{proof}

\subsection{Albanese variety and invariant classes}\label{par:Albanese}$\,$

\vspace{0.16cm}

The Picard variety $\Pic^0(M)$ is the torus $H^1(M,{\mathcal{O}})/H^1(M,\Z)$ 
parametrizing line bundles with trivial first Chern class (see \cite{Voisin:book}, 
section 12.1.3). 
If this torus is reduced to a point, the trivial line bundle is the unique line
bundle on $M$ with trivial first Chern class.  If its dimension is positive, the
first Betti number of $M$ is positive; in that case, the Albanese map provides
a morphism $\alpha:M\to A_M,$ where $A_M$ is the torus $H^0(M,\Omega^1_M)^*/H_1(M,\Z),$ of dimension   $b_1(M)/2$ (see \cite{Voisin:book}). This map satisfies a universal property with respect to morphisms from $M$ to tori. 

\begin{pro}\label{pro:pic0}
Let $M$ and $\Gamma$ be as in theorem \ref{thm:class-inv-div}.
If $\Pic^0(M)$ has positive dimension, the Albanese map  $\alpha:M\to A_M$ is a birational
morphism. 
\end{pro}

\begin{rem} When $M$ is a torus $\C^3/\Lambda,$ $\Gamma$ acts  
by affine transformations on it, because every automorphism
of a torus lifts to an affine transformation of its universal cover $\C^3.$
\end{rem}

\begin{proof}
If $\Pic^0(M)$ has positive dimension, the image of the Albanese map $\alpha$ 
has dimension $1,$ $2,$ or $3.$ From the universal property of the Albanese
variety, $\Gamma$ acts on $A_M$ and $\alpha$ is $\Gamma$-equivariant. Corollary
\ref{cor:no-inv-fib} implies that the image of $\alpha$ has dimension $3,$ so that $\alpha$
has generically finite fibers. 

Let $V$ be the image of $\alpha.$ Either $V$ has general type, or $V$ is
degenerate, i.e. $V$  is invariant under translation by 
a subtorus $K\subset A$ with $\dim(K)\geq 1$  (see \cite{Debarre:book}, chapter VIII). In the first
case, the action of $\Gamma$ on $V,$ and then on $M,$ would factor through a
finite group, a contradiction. In the second case, we may assume that $K$ is
connected and $\dim(K)$ is
maximal among all subtori preserving $V.$ Then, 
$\Gamma$  preserves the fibration of $V$ by orbits of $K.$ From corollary \ref{cor:no-inv-fib}, we deduce
that $V$ coincides with one orbit, so that $V$ is isomorphic to the torus $K.$ 
Since $V$ is a torus, it coincides with the Albanese variety $A_M$ and  $\alpha$
is a ramified covering. Since $A_M$ does not contain any projective plane, the
ramification locus of $\alpha$ is empty (theorem \ref{thm:class-inv-div}), and $\alpha$ is a
birational morphism. \end{proof}

\begin{cor}\label{cor:det-hom-class}
If $Z$ is an effective divisor, the homology class of which is $\Gamma$-invariant, 
then $Z$ is uniquely determined by its homology class and $Z$ is $\Gamma$-invariant. 
\end{cor}

\begin{proof}
Let $Z$ be a divisor, and let $L$ be the line bundle associated to $Z.$ A divisor $Z'$ is
the  zero set of a holomorphic section of $L$ if and only if $Z'$ is linearly equivalent
to $Z.$ Since $\Pic^0(M)$ is trivial (proposition \ref{pro:pic0}), the line bundle $L$ is
determined by the homology class $[Z].$ 

If $\Gamma$ preserves $Z,$ then $\Gamma$ linearly permutes the sections of
$L.$  Let now $\Phi_L:M\dasharrow \P(H^0(M,L)^*)$ be the rational map defined by $L.$ 
This map is $\Gamma$ equivariant: $\Phi_L\circ\gamma=\, ^t\!\eta(\gamma)\circ\Phi_L$
where $\eta(\gamma)$ denotes the linear action of $\gamma \in \Gamma$ on 
the space of sections of $L.$ From corollary \ref{cor:no-inv-fib}, we know that the
image of $\Phi_L$ has dimension $0$ or $3,$ and since there are elements
in $\Gamma$ with degrees $d_p(\gamma)>1,$ the dimension must be $0$
(see the proof of theorem \ref{thm:class-inv-div}). 
This implies that $L$ has a
unique section up to a scalar factor, which means  that the divisor $Z$ is uniquely
determined by its homology class. This implies that $Z$ is $\Gamma$-invariant.
 \end{proof}
 
\subsection{Contracting invariant surfaces}\label{par:cis}$\,$

\vspace{0.16cm}

From section \ref{par:InvSurf}, we know that every $\Gamma$-invariant or periodic
surface $S\subset M$ is a disjoint union of copies of the projective plane $\P^2(\C).$ 
Let $S_i,$ $i=1, ..., k,$ be $\Gamma$-periodic  planes, and $O_j,$ $j=1, ...,l,$ be the orbits of these planes. If the number $l$ is bigger than the dimension of $H^2(M,\Z),$ there is a linear relation between 
their cohomology classes $[O_i],$ that can be written in the form
$$
\sum_{i\in I} a_i [O_i]= \sum_{j\in J}b_j [O_j]
$$
where the sets of indices $I$ and $J$ are disjoint and the coefficients
$a_i$ and $b_j$ are positive integers. We obtain two distinct divisors 
$\sum a_i O_i$ and $\sum b_j O_j$ in the same invariant cohomology
class, contradicting corollary  \ref{cor:det-hom-class}.
This contradiction  proves the following assertion. 

\begin{thm}
Let $M$ and $\Gamma$ be as in theorem \ref{thm:class-inv-div}.
The number of $\Gamma$-periodic irreducible surfaces $S\subset M$ is finite. 
All of them are smooth projective planes with negative normal bundle, and these
surfaces are pairwise disjoint. There
is a birational morphism $\pi:M \to M_0$ to an orbifold $M_0$ with quotient
singularities which contracts simultaneously all $\Gamma$-periodic irreducible
surfaces, and is a local biholomorphism in the complement of these surfaces.
In particular, the group $\Gamma$ acts on $M_0$ and $\pi$ is $\Gamma$-equivariant
with respect to the action of $\Gamma$ on $M$ and $M_0.$
\end{thm} 

\begin{rem}
Of course, this result applies when $\Gamma$ is a lattice in an almost simple Lie group of rank at least $2.$ 
\end{rem}

%
%

\section{Actions of Lie groups}\label{par:actionsofliegroups}$\,$

%
%

In this section, we show how theorem A implies theorem B. By assumption, 
$\Gamma$ is a lattice in an almost simple higher rank Lie group $G,$
and $\Gamma$ acts on a connected compact K\"ahler manifold $M.$ 
Since this action has infinite image, it has finite kernel (\S \ref{par:normalsubgroups}).
Changing $\Gamma$ in a finite index subgroup, we can assume that this
action is faithful. Now, theorem $A$ implies that either the action is of Kummer
type, and then $M$ is birational to a Kummer orbifold, or the image of
$\Gamma$ is virtually contained in $\Aut(M)^0.$ We therefore assume that
$\Gamma$ embeds into $\Aut(M)^0.$

According to Lieberman and Fujiki, $\Aut(M)^0$ acts on 
the Albanese variety of $M$ by translations (see section \ref{par:Albanese}), and the kernel of this action is a linear algebraic group $L$; thus, $\Aut(M)^0$ is an extension 
of a compact torus by a linear algebraic group $L$ (see \cite{Lieberman:1978} and 
\cite{Campana-Peternell:survey} for these results). 
Since $\Gamma$ has property (T),
the projection onto the torus has finite image; once again, changing
$\Gamma$ into a finite index subgroup, we may assume that the image
of $\Gamma$ is contained in $L.$ As explained in \cite{Cantat:ENS}, section 3.3
(see also \cite{Zimmer:1984}, \S 3, or  \cite{Margulis:Book}, chapter VII and VIII), this implies 
that there is a non trivial morphism $G\to L.$ As a consequence, $L$ contains a simple 
complex Lie group $H,$ the Lie algebra of which has the same Dynkin diagram as $\g.$ 
The problem is now to classify compact K\"ahler threefolds $M$ with a holomorphic 
faithful action of an almost simple complex Lie group $H,$ the rank of which is larger
than $1.$ Hence, theorem B is a consequence of the following proposition

\begin{pro}\label{pro:LieGroupsActions}
Let $H$ be an almost simple complex Lie group with rank $\rk(H)\geq 2.$ Let $M$
be a compact K\"ahler threefold. If there exists an injective morphism $H\to \Aut(M)^0,$ then
$M$ is one of the following: 
\begin{itemize}
\item[(1)] a projective bundle $\P(E)$ for some rank $2$ vector bundle $E\to \P^2(\C),$ and then 
$H$ is locally isomorphic to $\PGL_3(\C)$;

\item[(2)] a principal torus bundle over $\P^2(\C),$ and $H$ is locally isomorphic to $\PGL_3(\C)$;

\item[(3)] a product $\P^2(\C)\times B$ of the plane by a curve of genus $g(B)\geq 2,$ and then 
$H$ is locally isomorphic to $\PGL_3(\C)$;

\item[(4)] the projective space $\P^3(\C),$ and $H$ is locally isomorphic to a subgroup of
$\PGL_4(\C),$ so that its Lie algebra is either $\sll_3(\C)$ or $\sll_4(\C)$;

\end{itemize}
\end{pro}

\begin{eg}
The group $\GL_3(\C)$ 
acts on $\P^2(\C),$ and its action lifts to an action on the total space of the line bundles
${\mathcal{O}}(k)$ for every $k\geq 0$; sections of ${\mathcal{O}}(k)$ are in one-to-one
correspondence with homogenous polynomials of degree $k,$ and the action of
$\GL_3(\C)$ on $H^0(\P^2(\C), {\mathcal{O}}(k))$ is the usual action on homogenous 
polynomials in three variables. 
Let $p$ be a positive integer and $E$ the vector bundle of rank $2$ over $\P^2(\C)$
defined by $E={\mathcal{O}}\oplus {\mathcal{O}}(p).$ Then $\GL_3(\C)$ acts
on $E,$ by isomorphisms of vector bundles. From this we get an action on the projectivized
bundle $\P(E),$ i.e. on a compact K\"ahler manifold $M$ which fibers over $\P^2(\C)$
with rational curves as fibers. 
A similar example is obtained from the $\C^*$-bundle associated to ${\mathcal{O}}(k).$
Let $\lambda$ be a complex number with modulus different from $0$ and $1.$ The quotient 
of this $\C^*$-bundle by multiplication by $\lambda$ along the fibers is a compact K\"ahler threefold,
with the structure of a torus principal bundle over $\P^2(\C).$ Since multiplication by 
$\lambda$ commutes with the $\GL_3(\C)$-action on  ${\mathcal{O}}(k),$ we obtain a (transitive) 
action of $\GL_3(\C)$ on this manifold. 
\end{eg}

Proposition \ref{pro:LieGroupsActions} is easily deduced from the classification of homogenous complex manifolds of
dimension at most $3,$ as described in the work of Winkelmann (see \cite{Winkelmann:LNM}).
Let us sketch its proof.

\begin{proof}[Sketch of the proof]
First, $H$ contains a Zariski dense lattice with property (T), so that we can apply
the results on invariant analytic subsets from section \ref{chap:invariant} to the group  $H.$ 
If $H$ has a fixed point $p,$ 
one can locally linearize the action of $H$ in a neighborhood of $p$ (see for example 
theorems 2.6 and 10.4 in \cite{Cairns-Ghys:1997}). This provides a regular morphism $H\to \SL_3(\C).$ 
Since all complex Lie subalgebras of $\sll_3(\C)$ with rank $\geq 2$ are equal to $\sll_3(\C),$
this morphism is onto. In particular, if $H$ has a fixed point, $H$ has an open orbit. 

Let us first assume that $H$ does not have any open orbit. In particular, no orbit of $H$
is a point. Let $O$ be an orbit of $H.$ Its closure is an invariant analytic subset; 
since it can not be a point, it must be a projective plane (corollary \ref{cor:allp2}). The action of
$H$ on $\overline O$ gives a map $H\to \Aut(\P^2(\C))=\PGL_3(\C),$ and this map 
is surjective because the rank of $H$ is at least $2.$ Hence, $H$ does not
preserve any strict subset of $\overline O,$ and $O$ coincides with its closure. 
From this follows that all orbits of $H$ are closed, isomorphic to $\P^2(\C),$
and the action of $H$ on each orbit coincides with the action of $\PGL_3(\C)$ on
$\P^2(\C).$
In that case, there is an invariant fibration $\pi:M\to B,$ where $B$ is a curve, with orbits
of $H$ as fibers. Let $A$ be a generic diagonal matrix in $\PGL_3(\C).$ The action
of $A$ on every fiber of $\pi$ has exactly three fixed points: One saddle, one sink and one source.
This gives three sections of the fibration $\pi:M\to B.$ The action of $H$ is transitive along the
fibers, and permutes the space of sections of $B.$ From this follows easily that the fibration 
is trivial. According to the value of the genus of $B,$ this case falls in one of the three
possibilities $(1),$ $(2),$ $(3).$

Let us now assume that $H$ has an open orbit $M_0\subset M,$ which does not coincide with $M.$
According to section \ref{chap:invariant}, its complement is a disjoint union of points and
of projective planes, $H$ is locally isomorphic to $\PGL_3(\C),$ and acts as $\PGL_3(\C)$ 
on each of these invariant planes. 
The open orbit $M_0$ is a homogenous complex manifold of dimension $3.$ Chapter 4 of \cite{Winkelmann:LNM} shows that $M_0$ is a $\C^*$-bundle over $\P^2(\C),$ and 
it follows that $M$ falls in case $(1)$ of the classification. 

The last case corresponds to transitive actions, when $M$ is isomorphic to a quotient
$H/N$ where $N$ is a closed subgroup of $H.$ A classical result due to Tits (see \cite{Tits:1962})
classifies all homogenous compact complex manifolds of dimension $3.$ For compact
K\"ahler threefolds with a transitive action of an almost simple Lie group, the list reduces to
the projective space $\P^3(\C),$ the smooth quadric $Q_3\subset \P^4(\C),$
and principal torus bundles as in $(2).$ The group
of automorphisms of $Q_3$ is locally isomorphic to $\SO_4(\C).$ Since $\so_4(\C)$
is isomorphic to $\sll_2(\C)\oplus \sll_2(\C),$ it does not contain any simple complex subalgebra
with rank $\geq 2.$ As a consequence, homogenous manifolds fall in case $(2)$ and $(4).$   \end{proof}

%
%

\section{Hodge structures and higher rank Lie groups}\label{chap:HSHRLG}

%
%

In this section, we use Hodge theory, Lie theory and Margulis rigidity to rule out 
several kind of lattices and Lie groups. In what follows, $G$ is 
a connected semi-simple Lie group with finite center,
without non trivial compact factor, and with  rank at least $2.$ The Lie
algebra of $G$ is denoted by $\g.$ 
Let $\Gamma$ be an irreducible lattice in $G$ and $\Gamma\to \Aut(M)$ 
be an almost  faithfull representation of $\Gamma$ into the group
of automorphisms of a compact K\"ahler threefold $M.$ 
The first statement that we want to prove is the following.

\begin{thm}\label{thm:LGC}
If the action of $\Gamma$ on the cohomology of $M$ does not
factor through a finite group, then $G$ is locally isomorphic to 
$\SL_3(\R)$ or $\SL_3(\C).$ 
\end{thm}

The proof is given in sections \ref{part:preliminaries} to \ref{par:CLG}. 

\subsection{Preliminaries}\label{part:preliminaries}$\,$

\vspace{0.16cm}

In order to prove theorem \ref{thm:LGC}, we first apply section 
\ref{par:normalsubgroups}: Since
the action of $\Gamma$ on $H^*(M,\Z)$ does not factor through a finite group, this action is almost faithful. 
Let $W$ denote $H^{1,1}(M,\R).$
From section \ref{par:dynadegrees},
we also know that the action of $\Gamma$ on $W$ is faithful, with discrete image.
Let us now apply corollary \ref{cor:margulis}, proposition
 \ref{pro:ExtensionCohomology}, and lemma \ref{lem:ExtensionCohomology}. 
The action of $\Gamma$ 
on $H^*(M,\R)$ extends virtually to a linear  representation 
of $G$ on $H^*(M,\R)$ that preserves the Hodge decomposition, 
the cup product, and Poincar\'e
duality.   
Hence, theorem  \ref{thm:LGC} is a corollary of the following
proposition and of Hodge index theorem.

\begin{pro}\label{pro:LGC}
Let $\g$ be a semi-simple Lie algebra, and $\g\to \End(W)$ a faithful finite dimensional
representation of $\g.$ If
\begin{itemize}
\item[(i)] there exists a symmetric bilinear $\g$-equivariant
mapping $\wedge : W \times W \to W^*,$ where $W^*$ is the dual representation, 
\item[(ii)] $\wedge$ does not vanish identically on any subspace of dimension $2$ in $W,$ and
 \item[(ii)] the real rank $\rk_\R(\g)$ is at least $2,$
\end{itemize} 
then $\g$ is isomorphic to $\sll_3(\R)$ or $\sll_3(\C).$ 
\end{pro}

Our main goal now, up to section \ref{par:CLG}, is to prove this proposition; $\g$
will be a semi-simple Lie algebra acting on $W,$ and
 $\wedge$ a $\g$ equivariant bilinear map with values in the dual 
 representation $W^*,$ as in the statement of proposition~\ref{pro:LGC}.

\subsection{$\sll_2(\R)$-representations}\label{par:sl2}$\,$

\vspace{0.16cm}

For all positive integers  $n,$ the Lie algebra $\sll_2(\R)$ acts linearly on the space of degree $n$ homogeneous 
polynomials in two variables. Up to isomorphism, this representation is the unique irreducible
linear representation of $\sll_2(\R)$ in dimension $n+1.$ The weights of this representation with 
respect to the Cartan subalgebra of diagonal matrices are 
$$
-n,  -n+2, -n+4, ..., n-4, n-2, n,
$$
and the highest weight $n$ characterizes this irreducible representation.  These representations
are isomorphic to their own dual representations. 

 \begin{lem}\label{lem:sl2}
 Let $\mu : \sll_2(\R)\to \g$ be an injective morphism of Lie algebras. Then
  the highest weights of the representation 
 $$
 \sll_2(\R)\to \g\to \End(W) 
 $$
 are bounded from above by $4,$ and the weight $4$ appears at most once.  
 \end{lem}

\begin{proof}
Let $V$ be an irreducible subrepresentation of $\sll_2(\R)$ in $W,$ 
and let $m$ be its highest weight. Let us assume that $m$ is the highest
possible weight among all irreducible subrepresentations of $W.$ Let
$u_m$ and $u_{m-2}$ be elements in $V\setminus \{0\}$ with respective
weights $m$ and $m-2.$ From Hodge index theorem (cf. lemma \ref{lem:hit}),
we know that one of the cup products
$$
u_m\wedge u_m, \,  \, u_m\wedge u_{m-2}, \,  \, u_{m-2}\wedge u_{m-2}
$$
is different from $0.$ The  weight of this vector is at least $2(m-2),$
and is bounded from above by the highest weight of $W^*,$ that is by $m$ ; 
consequently, $2(m-2)\leq m,$ and $m\leq 4.$ 

Let us now assume that the weight $m=4$ appears twice. Then there are 
$2$ linearly independant vectors $u_4$ and $v_4$  of weight 
$4.$ Since the highest weight of $W^*$ is also $4,$ the cup products
$u_4\wedge u_4,$ $u_4\wedge v_{4},$ and $v_{4}\wedge v_{4}$ vanish, contradicting 
Hodge index theorem (lemma \ref{lem:hit}).
\end{proof}

\subsection{Actions of $\sll_2(\R)\times \sll_2(\R)$}

\begin{lem}\label{lem:nosl2sl2}
If there is a faithful representation  $\g\to \End(W)$ as
in proposition \ref{pro:LGC}, then $\g$ does not 
contain any copy of $\sll_2(\R)\oplus \sll_2(\R).$ 
\end{lem} 

\begin{proof}
Let us assume that  $\g$ contains $\g_1\oplus \g_2$ with both $\g_1$ and $\g_2$ isomorphic to $\sll_2(\R)$.
Let $n_1$ be the highest weight of $\g_1$ and $n_2$ be the highest
weight of $\g_2$ in $W.$ Since the representation of $\g$ is faithful, both $n_1$ and $n_2$ are
positive integers. After permutation of $\g_1$ and $\g_2,$ we shall assume
that $n_1=\max(n_1,n_2).$

Let ${\mathfrak{h}}\leq \g_1\oplus \g_2$ be the diagonal 
copy of $\sll_2(\R)$ ; the highest weight of ${\mathfrak{h}}$ is $n_1.$
Let $u_i\in W$ be a vector of weight $n_i$ for
$\g_i.$ Since $\g_1$ and $\g_2$ commute, $u_1$ is not
colinear to $u_2.$  
If $u_i\wedge u_j$ is not zero, its weight for ${\mathfrak{h}}$ is $n_i+n_j.$ This implies
that $u_1\wedge u_1=0$ and $u_1\wedge u_2=0.$ 
Since $n_2$ is a highest weight for $\g_2,$ we also know that 
$u_2\wedge u_2=0$ because $2n_2$ does not appear as a
weight for $\g_2$ on $W^*.$ Hence, $\wedge$ should 
vanish identically on the vector space spanned by $u_1$ and $u_2,$
contradicting Hodge index theorem. This concludes the proof.
\end{proof}

\subsection{Actions of $\sll_3(\R)$}\label{par:cohosl3}$\,$

\vspace{0.16cm}

We now assume that $\g$ contains a copy of the Lie algebra $\sll_3(\R),$
and we restrict the faithful representation $\g\to \End(W)$ to $\sll_3(\R).$ 
Doing that, we  assume in this section that $\g=\sll_3(\R).$
In what follows, we choose the diagonal subalgebra 
$$
\aa = \left\{  \left( \begin{array}{ccc} a_1 & 0 & 0 \\ 0 & a_2 & 0 \\ 0 & 0 & a_3 \end{array}
\right) \, ; \quad a_1 + a_2 + a_3 =0\right\}
$$ 
as a Cartan  algebra in $\g.$  We shall denote by $\lambda_i: \aa \to \R$ the linear forms $\lambda_i(a_1,a_2,a_3) = a_i,$ 
and  by $\aa^+$ the Weyl chamber $a_1\geq a_2 \geq a_3.$ 
With such a choice, highest weights are linear forms on $\aa$ of type
$$
\lambda=a\lambda_1 - b\lambda_3,
$$
where $a$ and $b$ are non negative integers. Irreducible representations are
classified by their highest weight. For example, $(a,b)=(0,0)$ corresponds
to the trivial $1$-dimensional representation (denoted $T$ in what follows), $(a,b)=(1,0)$ corresponds
to the standard representation  on the vector space $E=\R^3,$ while $(0,1)$ is the
dual representation $E^*.$ When $(a,b)=(0,2),$ the representation is 
$\Sym_2(E^*),$ the space of quadratic forms on $E.$ The weight 
$(a,b)=(1,1)$ corresponds to the adjoint representation $(E\otimes E^*)_0$, i.e. to the action 
of $\SL_3(\R)$ by conjugation 
on $3\times 3$ matrices with trace zero. 
We refer to \cite{Fulton-Harris:book} for highest weight theory and a detailed description 
of standard representations of $\SL_3(\R).$  If $V$ is a representation of a Lie
group $G,$ we shall denote by $V^k$ the direct sum of $k$ copies of $V.$

\begin{pro} If $\g$ is isomorphic to $\sll_3(\R),$ then:
\begin{itemize}
\item[a.-]  The possible highest weights of irreducible subrepresentations in $W$ are $(0,0),$  $(1,0),$ $(0,1),$  $(1,1),$ $(2,0),$ $(0,2)$; 
\item[b.-] $W$ contains at most one irreducible subrepresentation with highest weight $a\lambda_1-b \lambda_3$
such that $a+b=2.$
\end{itemize}
\end{pro}

\begin{proof}
The action of $\sll_2(\R)$ on the space of quadratic homogenous
polynomial $ax^2+bxy +c y^2$ 
provides an embedding $\mu_2:\sll_2(\R)\to \g.$ If we restrict
it to the chosen Cartan subalgebras, the diagonal matrix
$diag(s,-s)$ is mapped to $diag(2s, 0, {-2}s).$
Let $\lambda=a\lambda_1-b\lambda_3$ be a highest weight for the representation $W.$
After composition with $\mu_2,$ we see that $2(a+b)$ is one of the weight of $\sll_2(\R).$
From lemma \ref{lem:sl2},  we get $a+b\leq 2.$ The list 
of possible weights follows. Point (b.-) follows from the last assertion in lemma
\ref{lem:sl2}. \end{proof}

\begin{lem} If $\g$ is isomorphic to $\sll_3(\R),$ then:
\begin{itemize}
\item[a.-] The representation $W$ contains at least one weight $a\lambda_1-b\lambda_3$ with $a+b=2$;
\item[b.-] The adjoint representation (corresponding to $(a,b)=(1,1)$) does not appear in 
$W$;
\item[c.-] If $W$ contains a factor of type $\Sym_2(E^*)$ (resp. $\Sym_2(E)),$ then $W$ 
does not contain any factor of type $E^*$ (resp. $E$). 
\end{itemize}
\end{lem}

\begin{proof}
The cup product defines a symmetric, $\g$-equivariant, bilinear map $\wedge: W\times W \to W^*.$
Let us interprete it as a linear map from $\Sym_2(W)$ to $W^*,$ and deduce 
the lemma from its study.

\vspace{0.16cm}

{\noindent}a.- If $W$ does not contain any weight with $a+b=2,$ then $
W$ is isomorphic to a direct sum 
$$
E^k \oplus (E^*)^l \oplus T^m 
$$
where $T$  is the trivial one-dimensional representation of $\g$ and $k,$ $l,$ and $m$
are non-negative integers. If $k>0,$ the cup product determines a $\g$-equivariant
linear map $\wedge_E: \Sym_2(E)\to W^*.$ Since $\Sym_2(E)$ is irreducible and is not isomorphic to 
$E,$ $E^*,$ or $T,$ this map $\wedge_E$ vanishes identically, contradicting  lemma \ref{lem:hit}.
From this we deduce that $k= 0,$ and similarly that $l=0$; this means that $W$ is the trivial representation,
contradicting the fact that $\g\to \End(W)$ is a faithful representation.

\vspace{0.16cm}

{\noindent}b.- The weights of the adjoint representation $(E\otimes E^*)_0$ are $\lambda_i-\lambda_j$ where $i$, $j$ $\in \{1,2,3\}$ 
are distinct numbers. This representation is self-dual. If $W$ contains it as a factor, then $W^*$ also,
and this factor is the unique one with highest weight $a\lambda_1-b\lambda_3$ such that $a+b\geq 2.$ Let $u_{ij}$
be a non-zero eigenvector of $\aa$ corresponding to the weight $\lambda_i-\lambda_j.$ Then $u_{13}\wedge u_{13}$
has
weight $2(\lambda_1-\lambda_3),$ which does not appear in $W^*.$ As a consequence, $u_{13}\wedge u_{13}=0,$
and similarly $u_{23}\wedge u_{23} = u_{13}\wedge u_{23}=0.$ This implies that $\wedge$ vanishes on
the two-dimensional vector space ${\text{Vect}}(u_{13}, u_{23}),$ contradicting lemma \ref{lem:hit}.

\vspace{0.16cm}

{\noindent}c.- Assume that $W$ decomposes as 
$$
\Sym_2(E) \oplus E^k \oplus  (E^*)^l  \oplus T^m,
$$
with $k>0.$ Choose $u$ in $\Sym_2(E)\setminus \{0\}$ a vector of weight $2 \lambda_1,$
and $v$ in $E\setminus \{0\}$ of weight $\lambda_1.$ The weights of $u \wedge u,$ $u \wedge v,$ and
$v\wedge v$ are equal to $n\lambda_1,$ with $n=4,$ $3,$ and $2$ respectively. None of them appears in $W^*,$ 
and lemma \ref{lem:hit} provides the desired contradiction. 
 \end{proof}
 
\begin{thm} If the Lie algebra $\g$ is isomorphic to $\sll_3(\R),$ and if $\g\to \End(W)$ 
is a linear representation as in proposition \ref{pro:LGC}, 
there exist two integers $k\geq 0$ and $m\geq 0$ such that the representation $W$ 
is isomorphic to 
$
\Sym_2(E^*) \oplus E^k \oplus T^m,
$
or to its dual.
\end{thm}

\subsection{Actions of $\sll_3(\C)$ and $\sll_3(\H)$}$\,$

\vspace{0.16cm}

Let us now assume that $\g$ is isomorphic to $\sll_3(\C).$ 
Let $E_\C$ denote the vector space $\C^3$ with its standard action of the
Lie algebra $\sll_3(\C).$ We denote by $\overline{E_\C}$ the complex conjugate representation,
and by $\Her(E_\C^*)$ the representation of $\sll_3(\C)$ on the space of hermitian forms.
As before, $T$ will denote the trivial one dimensional representation of $\sll_3(\C).$ 

\begin{thm}\label{thm:sl3C} If the Lie algebra $\g$ is isomorphic to $\sll_3(\C),$ and if $\g\to \End(W)$ 
is a linear representation as in proposition \ref{pro:LGC}, 
there exist three integers $k\geq 0,$ $l\geq 0,$ and $m\geq 0$ such that the representation $W$ 
is isomorphic to 
$$
W=\Her(E_\C^*))\oplus E_\C^k\oplus{\overline{E_\C}}^l \oplus T^m
$$
after composition by an automorphism of $\g.$
\end{thm}

\begin{rem}
The automorphism $b\mapsto  -\, ^t\! b$ of $\sll_3(\C)$ turns a representation into its dual, 
and $b\mapsto {\overline{b}}$ into its complex conjugate. If $\g$ is not
isomorphic to $\sll_3(\C)$ but contains a subalgebra $\g' \simeq \sll_3(\C),$
theorem \ref{thm:sl3C} holds for $\g '.$
\end{rem}

\begin{proof}
Let ${\mathfrak{h}} \subset\g$ be the Lie algebra $\sll_3(\R).$ 
With respect to the action of ${\mathfrak{h}},$ we may assume that $W$ splits as
$$
W=\Sym_2(E^*)\oplus E^k\oplus T^m.
$$
Let $\aa$ be the real diagonal subalgebra in $\g$: This is a Cartan subalgebra for
both ${\mathfrak{h}}$ and $\g.$ The highest weight for the representation of $\aa$ 
on $W$ is 
$$-2\lambda_3:(a_1,a_2,a_3)\mapsto -2a_3.$$
The eigenspace of $\aa$ corresponding to this weight has dimension $1$; 
it is spanned by the element $dx_3\otimes d{\overline{x_3}}$
of $\Her(E_\C^*).$  Its orbit under ${\mathfrak{h}}$ 
spans $\Sym_2(E^*),$ and one easily checks that its orbit under $\g$ 
determines a unique copy of the representation $\Her(E_\C^*).$ 
Since $\g$ is a simple Lie algebra, 
$W$ decomposes as a $\g$-invariant direct sum $W=\Her(E_\C^*)\oplus W_0$
where $W_0$ is isomorphic to $E^{k-1}\oplus T^m$ as a 
${\mathfrak{h}}$-module. In particular, all weights of $\aa$ on $W_0$ are equal
to $0$ or $L_1.$ This shows that $W_0$ is isomorphic to 
$$
E_\C^{k'}\oplus {\overline{E_\C}}^{l'}\oplus T^m
$$
as a $\g$-module, with $k=2(k'+l')+1.$ 
\end{proof}

Let $\H$ be the field of quaternions. 

\begin{pro}\label{pro:nosl3H}
If there is a faithful representation  $\g\to \End(W)$ as
in proposition \ref{pro:LGC}, then $\g$ does not 
contain any copy of  $\sll_3(\H).$ 
\end{pro}

\begin{proof}
Let us switch to representations of Lie groups, and 
assume that $\SL_3(\H)$ acts on $W,$ as in proposition \ref{pro:LGC}. 
Let $\Delta$ be the subgroup $\SL_3(\H)$ which is made of diagonal
matrices with coefficents $h_1,$ $h_2,$ and $h_3$ in $\H$ (with the
condition that the  real determinant of the matrix is $1$, see \cite{Fulton-Harris:book} page 99). 

Let $A= diag(a_1,a_2,a_3)$ be an element of $\Delta\cap \SL_3(\C)$ with complex entries $a_j=\rho_j e^{i\phi_j},$ where $\rho_j=\vert a_j\vert.$ Theorem \ref{thm:sl3C} implies that the eigenvalues of $A$ on $W\otimes \C$ are
\begin{itemize}
\item $\rho_j^{-2},$ $j=1,2,3,$ with multiplicity $1$;
\item $\rho_j e^{i(\phi_{j'}-\phi_{j"})},$ where $\{j,j',j"\}= \{1,2,3\},$ with multiplicity $1$;
\item $\rho_j e^{i\phi_j},$ $j=1,2,3,$ with multiplicity $k$;
\item  $\rho_j e^{-i\phi_j},$ $j=1,2,3,$ with multiplicity $l$;
\item $1$ with multiplicity $m.$
\end{itemize}
Let now $B$ be an element of $\SL_3(\H)$ of type $diag(e^{i\psi},1,1)$
(resp $diag(1, e^{i\psi}, 1),$ resp. $diag(1,1,e^{i\psi})$). Since
$B$ commutes with $\Delta\cap \SL_3(\C),$ $B$ preserves the eigenspaces of $A.$
In particular, $B$ preserves $\Her(E_\C^*),$ acts trivially on the one
dimensional eigenspaces of $A$ corresponding to the eigenvalues 
$\rho_j^{-2},$ and as a standard rotation on the invariant planes associated
to pairs of conjugate eigenvalues $\rho_j e^{\pm i(\phi_{j'}-\phi_{j"})}.$
In other words, $B$ acts on $\Her(E_\C^*)$ as elements of $\GL_3(\C)$ do. 
As a consequence, composing elements of type $B$ with elements
of type $A,$ we see that the action of $\Delta\cap \GL_3(\C)$ on 
$\Her(E_\C^*)$ is the standard action on hermitian forms; the previous
computation for the eigenspaces of matrices $A$ holds for this larger group.

Let us now choose $A=diag(\rho_1 e^{i\phi_1}, \rho_2, \rho_3 e^{-i\phi_3})$ 
with pairwise distinct $\rho_j,$ distinct irrational $\phi_1$ and $\phi_3,$ and
$\rho_1\rho_2\rho_3=1.$ This element of $\SL_3(\H)$ commutes with
the subgroup $S=\{diag(1,q,1)\, \vert \, q \in \SU_2\}$
of  $\SL_3(\H).$ In particular, $S$ preserves the eigenspaces of $A.$
But $S\simeq \SU_2$ does not have any non trivial representation in dimension 
$<3.$ This implies that $S$ acts trivially on $\Her(E_\C^*).$ 
In the same way, $S'=\{diag(1,1,q)\, \vert \, q \in \SU_2\}$ acts trivially on 
$\Her(E_\C^*).$ This is a contradiction because $diag(1,e^{i\phi}, e^{-i\phi}),$
$\phi\neq 0,$ is an element of the product $S  S'$ that does not act trivially on $\Her(E_\C^3).$
 \end{proof}

\subsection{Proof of proposition \ref{pro:LGC}}\label{par:CLG}

\begin{lem}\label{lem:CLG}
Let $\g$ be a simple real Lie algebra.
If the real rank of $\g$ is at least $3,$ then $\g$ contains a copy of 
$\sll_2(\R)\oplus \sll_2(\R).$ 
If the real rank of $\g$ is equal to $2,$ then
\begin{itemize}
\item[a.-]If $\g$ admits a complex structure, then $\g$ is isomorphic to one of the three 
algebras $\sll_3(\C),$ $\sp_4(\C)$  or $\g_2(\C).$
\item[b.-] If $\g$ does not admit a 
complex structure, either $\g$ contains a copy of $\sll_2(\R)\oplus\sll_2(\R),$ 
or $\g$ is isomorphic to one of the algebras
$\sll_3(\R), $
$\sll_3(\H), $ 
or $\e_{IV}.$
\end{itemize}
\end{lem}

\begin{proof}[Sketch of the proof] 
This is a consequence of the classification of simple Lie algebras 
(see \cite{Knapp:book}) and exceptional isomorphisms in small dimensions. 
For example, the complex Lie group $E_6$ has two real forms of 
rank~$2$: $\e_{III}$ and $\e_{IV}.$  The  root system of $\e_{III}$
with respect to its maximal $\R$-diagonalizable subalgebra is isomorphic
to $BC_2,$ which contains the root system of
 $\sll_2(\R)\oplus \sll_2(\R)$;  the same is true for $\g_2.$ 
Another example is given by the Lie algebras $\so_{2,k},$ $k\geq 2$:
All of them contain $\so_{2,2},$ i.e. $\sll_2(\R)\oplus\sll_2(\R).$ 
\end{proof}

\begin{lem}\label{lem:e4}
The real Lie algebra $\e_{IV}$ contains a copy of $\sll_3(\H).$
\end{lem}

\begin{proof}
This follows from the classification of simple real Lie algebras in terms
of their Vogan diagrams (see \cite{Knapp:book}, chapter VI), and from 
the fact that the diagram of $\sll_3(\H)$ embeds into the diagram of
$\e_{IV}.$
\end{proof}

\begin{proof}[Proof of proposition \ref{pro:LGC}] Let $\g$ and $\g\to \End(W)$
be as in proposition \ref{pro:LGC}.
If  $\g$ has two simple factors of rank $\geq 1,$ then $\g$ contains a copy 
of $\sll_2(\R)\oplus \sll_2(\R),$ and lemma \ref{lem:nosl2sl2} provides a
contradiction. We can therefore assume that $\g$ is simple. Lemma \ref{lem:nosl2sl2}
implies  that $\g$ does not contain any copy of 
$\sll_2(\R)\oplus\sll_2(\R),$ and proposition \ref{pro:nosl3H} shows that it does not contain any copy 
of $\sll_3(\H).$ The proposition is now a consequence of lemmas \ref{lem:CLG} and \ref{lem:e4}.
\end{proof}

%
%

\section{Lattices in $\SL_3(\R)$: Part I}\label{chap:SL3I}

%
%

We pursue the proof of theorem A. 
From section \ref{chap:HSHRLG}, theorem \ref{thm:LGC},
we can assume that $\Gamma$ is a lattice in $\SL_3(\R)$ or $\SL_3(\C).$
Here we deal with the first case, namely $G=\SL_3(\R).$ 
Our main standing assumptions are now 
\begin{itemize}
\item $\Gamma$ is a lattice in $G=\SL_3(\R)$;
\item $\Gamma$ acts holomorphically on a compact K\"ahler threefold $M$;
\item the action of $\Gamma$ on $H^*(M,\R)$ extends to a non trivial linear representation
of $G.$
\end{itemize}
We shall denote by $W$ both the cohomology group $H^{1,1}(M,\R)$ and
the linear representation $G\to \GL(H^{1,1}(M,\R)).$ This gives a representation 
of the Lie algebra
$\g=\sll_3(\R),$ and section \ref{par:cohosl3} shows that this representation $\g\to \End(W)$ 
decomposes into the direct sum
$$
(*) \quad W=\Sym_2(E^*)\oplus E^k\oplus T^m,
$$ 
or its dual; composing $\Gamma\to \Aut(M)$ with the automorphism $B\mapsto \, ^t\!(B^{-1})$
of $G,$ we shall assume that $W$ is isomorphic to the direct sum $(*)$. The representation $H^{2,2}(M,\R)$ is isomorphic to the dual of $W.$ We shall 
denote it by $W^*,$ with its direct sum decomposition 
$$
(**) \quad W^*=\Sym_2(E)\oplus (E^*)^k\oplus (T^*)^m
$$ 
where $T^*$ is just another notation for the one dimensional trivial representation of $G.$

Our goal is to show that the action of $\Gamma$ on $M$ is of Kummer type. In this 
section, we study the structure of the cup product $\wedge$ and the position
of the lattice $H^2(M,\Z)$ with respect to the Hodge decomposition of 
$H^2(M,\C),$  the direct sum decomposition $(*)$ of $H^{1,1}(M,\R),$ and
the K\"ahler cone of $M.$

Most of the lemmas that we shall prove in this section are not specific to 
$\SL_3(\R),$ and will be used for $G=\SL_3(\C)$ in section \ref{chap:SL3C}.

\subsection{The cup product}

\subsubsection{Intersections between irreducible factors}\label{par:sectionwedge}

The following lemmas are straightforward applications of representation theory, in the same spirit as what we did in section \ref{chap:HSHRLG}. 

\begin{lem}
Up to a multiplicative scalar factor,
\begin{itemize}
\item[a.-] there is a unique non-zero symmetric bilinear $G$-equivariant mapping from 
$\Sym_2(E^*)$ to $\Sym_2(E).$

\item[b.-] All $G$-equivariant bilinear symmetric mappings from $\Sym_2(E^*)$ 
to $(E^*)^k \oplus (T^*)^m$ vanish identically.
\end{itemize}
\end{lem}
 
\begin{proof}
From \cite{Fulton-Harris:book}, page 189, we know that 
$$
\Sym_ 2(Sym_2(E^*)) \simeq \Sym_4(E^*)\oplus \Sym_2(E).
$$
Both assertions follow from this isomorphism. \end{proof} 
  
In the same way, one shows that,
up to multiplication by a scalar factor, there is a unique $G$-equivariant, symmetric,
and bilinear map from $E$ to $\Sym_2(E),$ 
but there is no non-zero map of this type from $E$ to $(E^*)^k \oplus (T^*)^m.$ 

\begin{lem}[See \cite{Fulton-Harris:book}, pages 180-181]
The symmetric tensor pro\-duct
$$
\Sym_2(\Sym_2(E^*)\oplus E^k)
$$ 
decomposes as the
direct sum of the following factors :  $ \Sym_2(\Sym_2(E^*))$, $\Sym_2(E^k)$,
and $k$ copies of $\Gamma_{1,2}\oplus E \oplus E^*,$ 
where $\Gamma_{1,2}$ is the irreducible representation with highest
weight $\lambda_1-2\lambda_3.$ 
\end{lem}

As a consequence, there exist non trivial symmetric and $G$-equivariant 
bilinear mappings from $\Sym_2(E^*)\oplus E^k $ to $(E^*)^k ,$
but all of them vanish identically on $\Sym_2(E^*)$ and $E^k .$ 
In particular, 
\begin{enumerate}
\item if $u$ is in $\Sym_2(E^*)$ then   $u\wedge u \in \Sym_2(E)$;
\item if $v$ is in $E^k$ then   $v \wedge v \in \Sym_2(E)$;
\item if $u$ is in $\Sym_2(E^*)$ and $v$ is in $E^k$ then   $u \wedge v \in  (E^*)^k$;
\item if $t$ is in $T$ then $t\wedge t \in (T^*)^m$;
\item if $u$ is in $\Sym_2(E^*)\oplus E^k$ and  $t$ is in $T$ then $u\wedge t = 0.$
\end{enumerate}
Moreover, $\wedge$ is uniquely determined up to a scalar multiple once restricted 
to $\Sym_2(E^*)$ (resp. to each factor $E$ of $E^k$).

\subsubsection{Cubic form}\label{par:CubicForm}

Let $D$ be the cubic form which is defined on $W$ by 
$$
D(u) = \int_M u\wedge u \wedge u.
$$

\begin{lem}\label{lem:cubicform}
When restricted to $\Sym_2(E^*),$ the cubic form $D$ coincides with a 
non trivial scalar multiple of the determinant $\det:\Sym_2(E^*)\to \R.$ It vanishes identically on $E^k.$
\end{lem}

The automorphism $u\mapsto -u$ of $\Sym_2(E^*)$ commutes with the
action of $G$ and changes $\det$ in $-\det.$ As a consequence, we shall
assume that {\sl{there exists a positive number $\epsilon$ such that
$$
\int_M u\wedge u \wedge u = \epsilon \,\det(u) 
$$
for all $u$ in $\Sym_2(E^*).$ }}

\begin{proof}
The trilinear mapping $D$ is symmetric, and 
$$
\Sym_3(\Sym_2(E^*))= \Sym_6(E^*) \oplus \Gamma_{2,2}\oplus T,
$$
where the trivial factor $T$ is generated by $\det$ (see \cite{Fulton-Harris:book}, page 191). 
This implies that $D$ is proportional to $\det.$ Let $\kappa_0\in W$ 
be a K\"ahler class, and let $\kappa_0= u_0 + v_0 + t_0$ be its decomposition 
with respect to $W= \Sym_2(E^*) \oplus E^k  \oplus T^m .$ By Hodge index theorem,  $Q_\kappa(u,u) = -\int_M \kappa_0\wedge  u\wedge u$ 
does not vanish identically along a subspace of $W$ of dimension $>1$ 
(see the proof of lemma \ref{lem:hit}). This remark and property (1) above imply 
the existence of an element $u$ in $\Sym_2(E^*)$ such that 
$$
 \int_M u_0\wedge u \wedge u \neq 0.
$$
Since $D$ is symmetic, $D$ does not vanish identically along
$\Sym_2(E^*),$ and $D$ is a non zero scalar multiple of $\det.$

The second assertion follows
from the fact that $v\wedge v\in \Sym_2(E)$ for all $v$ in $E^k,$ 
and that $\int_M w\wedge e = 0$ for all $w$ in $ \Sym_2(E)\subset W^*$ 
and all $e$ in $E.$ \end{proof}
 
The following lemma shows how the hard Lefschetz theorem can be used in the
same spirit as what we did with Hodge index theorem.
 
\begin{lem}\label{lem:DonT}
There are elements $t\in T^m \subset W$ such that $D(t)\neq 0.$
\end{lem}

\begin{proof}
 Let $\kappa_0\in W$ 
be a K\"ahler class, and let $\kappa_0= u_0 + v_0 + t_0$ be its decomposition 
with respect to $W= \Sym_2(E^*) \oplus E^k  \oplus T^m .$
The hard Lefschetz theorem (see \cite{Voisin:book}, page 142 theorem 6.25) 
shows that the linear map 
$$
\beta \mapsto \kappa_0\wedge \beta
$$
is an isomorphism between $W$ and $W^*.$ As a consequence, $\kappa_0\wedge t \neq 0$ for all 
$t \in T\setminus\{0\}.$ Let $t_1$ be a non-zero element of $T^m.$ Then,
from property (5) above, we get
$$
  \kappa_0\wedge t_1= (u_0+v_0+t_0)\wedge t_1 = t_0\wedge t_1\in (T^m)^*\setminus\{0\}.
$$
Since $t_0\wedge t_1$ is different from $0,$ there exists $t_2$ in $T^m$ such that $t_0\wedge t_1 \wedge t_2\neq 0.$
This implies that the symmetric trilinear form $D$ does not vanish identically on $T^m.$
\end{proof}

\subsection{Cohomology with integer coefficients}$\,$

\vspace{0.16cm}

Our next goal is to describe the position of $W=H^{1,1}(M,\R)$ with respect to  $H^2(M,\Z).$ 
Note that $H^2(M,\Z)$ is a lattice in $H^2(M,\R),$ where we use the following
definition: A {\bf{lattice}} ${\mathcal{L}}$ in a real vector space $V$ is a cocompact discrete subgroup
of $(V,+).$ 

\subsubsection{Invariant lattices}

\begin{lem}\label{lem:Lindirectsum}
Let $G$ be an almost simple Lie group and $\Gamma$ be a lattice in $G.$
Let $V$ be a $G$-linear representation with no trivial factor, let $T$ be the 
trivial  one dimensional  representation of $G,$ and let $V\oplus  T^m$ be the direct sum of 
$V$ with $m$ copies of $T.$ If ${\mathcal{L}}\subset V\oplus T^m$ is a
$\Gamma$-invariant lattice, then ${\mathcal{L}}\cap V$ is also a lattice in $V.$
\end{lem}

\begin{proof}
Let $u=v_0+t_0$ be an element of ${\mathcal{L}},$ with $v_0$ in $V\setminus\{0\}$ and $t_0$ in $T^m.$ Let 
us consider the subset ${\mathcal{L}}_0$ of ${\mathcal{L}}\cap V$  defined by
$$
{\mathcal{L}}_0=\{v-w \, \vert \, w,v \in {\mathcal{L}}\cap (V+t_0)\}.
$$
This set is  $\Gamma$-invariant, and contains all elements of type $\gamma(v_0)-v_0,$ when 
$\gamma$ describes $\Gamma.$ Since the representation $V$
does not contain any trivial factor, and $\Gamma$ is Zariski dense in $G,$ we know
that $\Gamma$ does not fix the vector $v_0$; in particular, ${\mathcal{L}}_0$ spans a 
non trivial subspace $V_0$ of $ V.$ If $V_0$ coincides with $V,$ we are done, because
${\mathcal{L}}_0$ is a lattice in $V_0.$  Otherwise, the codimension of $V_0$ in $V$ is positive. 
Since ${\mathcal{L}}$ is a lattice in $V+T,$ the projection of ${\mathcal{L}}$ on $V$ spans $V,$ so that there exists an element $v_1+t_1$ in ${\mathcal{L}}$ such that $v_1$ is not in $V_0.$ The same argument, once applied to $v_1+t_1$ in place of $v_0+t_0,$ produces a new $G$-invariant subspace $V_1$ in $V$ for which ${\mathcal{L}}\cap V_1$ is a lattice in $V_1.$ Either
$V_0\oplus V_1 = V,$ and the proof is complete, or the construction can be pushed 
further. In less than $\dim(V)$ steps, we are done.  \end{proof} 

\begin{lem}
If ${\mathcal{L}}_0$ is a $\Gamma$-invariant lattice in $\Sym_2(E^*)\oplus E^k,$ 
the intersection ${\mathcal{L}}_0\cap \Sym_2(E^*)$ is a lattice in $\Sym_2(E^*).$ 
\end{lem}

\begin{proof}
Let $B$ be an element of $\Gamma.$ Since $\Gamma$ is an arithmetic lattice,
the eigenvalues of $B$ are algebraic integers. Let $\alpha$ be such an eigenvalue, 
and let $Q_\alpha$ be its minimal polynomial: By construction, $Q_\alpha$ 
is a polynomial in one variable with integer coefficients, and the roots of $Q_\alpha$
are all Galois-conjugate to $\alpha.$

Let us now assume that $B$ is diagonalizable (over $\C$) with three distinct eigenvalues satisfying
$\vert \alpha\vert > \vert \beta\vert >\vert \gamma\vert.$ Such elements exist because 
$\Gamma$  is Zariski dense in $G$ (see section \ref{par:proximality}).  Let $P_B$ be the product
$$
P_B(t)=Q_\alpha(t) Q_\beta(t) Q_\gamma(t);
$$
its coefficients $a_j$ are integers and $P_B(B)=\sum_j a_j B^j= 0.$ 

Let $\rho_1:G\to \GL(E^k)$ be the diagonal representation. Let  
$$
\rho_2:G\to \GL(\Sym_2(E^*))
$$ 
be the representation on quadratic forms, and $\rho=\rho_1\oplus \rho_2$ the direct sum of these two representations. Then 
$
P_B(\rho_1(B))=0.
$
Let us assume that $P_B(\rho_2(B))=0.$ The eigenvalues of $\rho_2(B)$ 
are roots of $P_B,$ and so are their Galois-conjugates. Let $\mu$ be a root of $P_B.$ 
Then, by construction, $\mu$ is conjugate to an eigenvalue of $B,$ so that $\mu^{-2}$ is conjugate to an eigenvalue 
of $\rho_2(B).$ As a consequence, $P_B(\mu^{-2})=0.$ From this we deduce the following: 
If $\mu$ is  a root of $P_B,$ so is $\mu^4.$ 
This implies that all roots of $P_B$ are roots of unity, contradicting the choice of $B.$
This contradiction shows that $P_B(\rho_2(B))$ is different from $0.$ 

Let us now choose an element $u_0+v_0$ in ${\mathcal{L}}_0$ with $u_0\in \Sym_2(E^*)$ and $v_0$ in $E^k.$ Then $P_B(\rho(B))(u_0+v_0)$ is an element of ${\mathcal{L}}_0$ because ${\mathcal{L}}_0$ is $\Gamma$-invariant and $P_B$ has integer coefficients. Moreover,
$$
P_B(\rho(B))(u_0+v_0)= P_B(\rho_2(B))(u_0)
$$
is an element of $\Sym_2(E^*).$  Since $P_B(\rho_2(B))\in \End(\Sym_2(E^*))$ is different from $0$ 
and ${\mathcal{L}}_0$ is a lattice, we can choose $u_0+v_0$ in such a way that $ P_B(\rho_2(B))(u_0)\neq 0.$
This implies that the lattice ${\mathcal{L}}_0$ intersects $\Sym_2(E^*)$ non trivially. 
From the Zariski density of $\Gamma$ in $G$ and the irreducibility of the linear
representation $\Sym_2(E^*),$ we conclude that ${\mathcal{L}}_0\cap \Sym_2(E^*)$ is a lattice
in $\Sym_2(E^*).$ \end{proof}

\subsubsection{The lattice $H^2(M,\Z)$}

Let us assume that $H^{2,0}(M,\C)$ is not trivial.  The sesquilinear mapping
$$
(\Omega_1, \Omega_2)\mapsto \Omega_1\wedge \overline{\Omega_2} ,
$$
from  $H^{2,0}(M,\C)$ to $H^{2,2}(M,\C)=W^*\otimes_\R \C$ is $G$-equivariant. 
Moreover, $\Omega\wedge \overline{\Omega}=0$  if and only if $\Omega=0.$ 
From this we deduce that, if $\lambda= a\lambda_1-b\lambda_3$ is a weight for the linear representation of $G$ on
$H^{2,0}(M,\C),$ then $2a\lambda_1-2b\lambda_3$ is a weight for $W^*.$  This implies that 
the linear representation of $G$ on $H^{2,0}(M,\C)\oplus H^{0,2}(M,\C)$ decomposes
as a sum of standard and trivial factors: There exist two integers $k'$ and $m'$ such
that 
$$
H^{2,0}(M,\C)\oplus H^{0,2}(M,\C)= E^{k'} \oplus T^{m'}.
$$
We can then write
the linear representation of $G$ on $H^2(M,\R)$ as a direct sum
$$
H^2(M,\R)= \Sym_2(E^*) \oplus E^{k''} \oplus T^{m''}
$$
where $k''= k+k'$ and $m''= m+m'.$

\begin{pro}\label{pro:invlattices}
The lattice $H^2(M,\Z)$ intersects each of the three factors 
$\Sym_2(E^*),$ $E^{k''},$ and $T^{m''}$ on 
lattices. 
\end{pro}

\begin{proof}
Let  ${\mathcal{L}}$ be the lattice $H^2(M,\Z)$ in $H^2(M,\R).$
We shall say that a subspace $V$ of $H^2(M,\R)$ is defined over the
integers, if it is defined by linear forms from the dual lattice ${\mathcal{L}}^*.$
From the two previous lemmas, we know that ${\mathcal{L}}$
intersects $\Sym_2(E^*)$ on a lattice ${\mathcal{L}}_1.$ The cup product is defined
over the cohomology with integral coefficients, so that $H^4(M,\Z)$ 
intersects also the subspace $\Sym_2(E)$ of $W^*$ onto a lattice. 
Since its orthogonal with respect to Poincar\'e duality coincides with $ E^{k''}  \oplus T^{m''},$ 
this subspace of $H^2(M,\R)$ is defined over
$\Z,$ and intersects ${\mathcal{L}}$ on a lattice. From lemma \ref{lem:Lindirectsum}, ${\mathcal{L}}\cap  E^{k''} $ is a lattice in $E^{k''} .$ Since $ T^{m''}$
is the set of cohomology classes $t$ such that $u \wedge t = 0$ for all
$u$ in $ \Sym_2(E) \oplus(E^*)^{k''},$ this subspace is also defined over
$\Z,$ and intersects ${\mathcal{L}}$ on a lattice. \end{proof}

\subsection{Invariant cones}
\subsubsection{Cones from complex geometry}\label{par:conesdefi}

Let us recall that a convex cone ${\mathcal{C}}$ in a real vector space $V$
is {\bf{strict}} when it does not contain any line. In other words, ${\mathcal{C}}$ 
is entirely contained on one side of at least one hyperplane of $V.$
The {\bf{K\"ahler cone}} of $M$ is the set $\Ka(M)\subset W$ of cohomology classes of
 all K\"ahler forms of $M.$
This set is an open convex cone;  its closure
${\overline{\Ka}}(M)$ is a strict and closed convex cone, the interior of which coincides
with $\Ka(M).$ We shall say that
${\overline{\Ka}}(M)$ is the cone of {\bf{nef}} $(1,1)$-cohomology classes. 
The {\bf{pseudo-effective cone}}  is the set ${\overline{\Eff}}(M)\subset W$ of cohomology
classes $[c]$ of closed positive currents of type $(1,1).$ This is a strict and closed 
convex cone, that contains ${\overline{\Ka}} (M).$ 
The {\bf{cone of numerically positive classes}} is the subset $\PP(M)\subset W$ of cohomology
classes $[\alpha]$ such that 
$$
\int_Y \alpha^{\dim(Y)}> 0
$$
for all analytic submanifolds $Y$ of $M.$ 
All these cones are invariant under the action of $\Aut(M).$  
In our context, they are $\Gamma$-invariant, but it is not clear a priori that 
$\Ka(M)$ or ${\overline{\Eff}}(M)$ is $G$-invariant.

\subsubsection{Invariant cones in $E^k\oplus T^m$}

\begin{lem}
The representation $E^k$ (resp. $(E^*)^k$) does not contain any 
non trivial $\Gamma$-invariant strict convex cone. 
\end{lem}

\begin{proof}
Let ${\mathcal{D}}\subset E$ be a $\Gamma$ invariant convex cone, which is different
from $\{0\}.$ Its boundary provides a closed $\Gamma$-invariant subset in the
projective plane $\P(E).$ Since the limit set of $\Gamma$ in $\P(E)$ coincides
with $\P(E)$ (see \S \ref{par:limitset}), the boundary of ${\mathcal{D}}$ is empty,
and ${\mathcal{D}}$ coincides with $E.$

Let now ${\mathcal{C}}\subset E^k$ be a $\Gamma$-invariant 
convex cone. Let $A$ be the smallest linear subspace of $E^k$ which
contains ${\mathcal{C}}.$ Since $\Gamma$ is Zariski dense in $G$ and preserves
$A,$ the group $G$ preserves $A.$ Replacing $E^k$ by $A,$ we can
therefore assume that ${\mathcal{C}}$ spans $E^k$ ; in other words, 
we shall assume that ${\mathcal{C}}$ has non empty interior in $E^k.$ 
Let ${\mathcal{C}}'$ be the interior of ${\mathcal{C}}.$ 

Let $\pi: E^k\to E$ be the projection onto the first factor.
 The convex cone $\pi({\mathcal{C}}')$  is
open and $\Gamma$-invariant; as such it must co\"{\i}ncide with $E.$ 
The fiber $\pi^{-1}\{ 0 \}\cap  {\mathcal{C}}'$ is an open, convex cone
in $E^{k-1},$ and is $\Gamma$-invariant. After $k-1$ steps, we end up 
with a $\Gamma$-invariant open convex subcone of  ${\mathcal{C}}'$  in $E.$ 
This cone must coincide with $E$; in particular,  ${\mathcal{C}}$ is not strict. \end{proof}

\begin{lem}\label{lem:coneE+T}
If ${\mathcal{C}}$ is a strict and $\Gamma$-invariant convex cone in $E^k  \oplus T^m $ (resp.  $(E^*)^k  \oplus (T^*)^m$) 
then ${\mathcal{C}}$ intersects $T^m \setminus \{ 0 \}$ (resp. $(T^*)^m$).
\end{lem}

\begin{proof}
Let ${\mathcal{D}}$ be the projection of ${\mathcal{C}}$ onto $E^k.$ 
The previous lemma shows the existence of a vector $v$ in $E^k$ such that
$v$ and $-v$ are both contained in ${\mathcal{D}}.$ From this follows the existence
of a pair $(t,t')$ of elements of $T^m$ such that both $v+t$ and $-v+t'$
are contained in ${\mathcal{C}}.$ Since ${\mathcal{C}}$ is convex and strict, 
$t+t'$ is contained in ${\mathcal{C}} \cap T^m \setminus \{ 0 \}.$
\end{proof}

\begin{rem}\label{rem:lem5.10}
If ${\mathcal{C}}$ is a strict, open, and $\Gamma$-invariant convex cone and ${\mathcal{L}}$ is a $\Gamma$-invariant lattice  in $E^k  \oplus T^m $ (resp.  $(E^*)^k  \oplus (T^*)^m$), the same proof implies
that ${\mathcal{C}}\cap {\mathcal{L}}$ intersects $T^m.$
\end{rem}

\subsubsection{Invariant cones in $\Sym_2(E^*)$}
Let us now study invariant cones in the representation $\Sym_2(E^*).$
Let $Q_+\subset \Sym_2(E^*)$ be the cone of positive definite quadratic
forms. This convex cone is open, strict, and $G$-invariant. 

\begin{pro}\label{pro:coneQ+carac}
If ${\mathcal{C}}$ is a $\Gamma$-invariant, strict, and open convex cone in
$\Sym_2(E^*),$ ${\mathcal{C}}$ co\"{\i}ncides with $Q_+$ or its opposite
$-Q_+.$ If ${\mathcal{C}}$ is a $\Gamma$-invariant, strict, and closed convex cone,
then ${\mathcal{C}}$ is either $\{ 0 \},$ $\overline Q_+,$ or $- \overline Q_+.$ 
\end{pro}

\begin{proof}
Let ${\mathcal{C}}\subset \Sym_2(E^*)$ be a $\Gamma$-invariant, strict convex cone
which is different from $\{0\}.$ 
First of all,  ${\mathcal{C}}$ spans 
$\Sym_2(E^*)$ because $\Gamma$ is Zariski dense in $G.$  In particular, the interior 
${\mathcal{C}}'$ of ${\mathcal{C}}$ is
not empty. 

The projective space $\P(\Sym_2(E^*))$ has dimension $5.$ The group 
$\Gamma$ acts on it by projective transformations, and its limit set  co\"{\i}ncides
 with the surface
$S\subset \P(\Sym_2(E^*))$ of projective classes of quadratic forms of 
rank $1.$ 
The projection $\P(\partial C)$ of the boundary of ${\mathcal{C}}$ onto the
projective space $\P(\Sym_2(E^*))$ is a closed and $\Gamma$-invariant subset of 
$\P(\Sym_2(E^*)).$ As such it must contain the limit set $S.$ 

The proposition is now a consequence of the following observation:
If a strict and closed convex cone ${\mathcal{D}}\subset \Sym_2(E^*)$  contains $S$ in its
projective boundary, then  ${\mathcal{D}}$ contain $Q_+$ or its opposite $-Q_+$ . \end{proof}

\subsubsection{The nef cone}

\begin{pro}\label{pro:Q+nef}
The intersection of the nef cone ${\overline{\Ka}}(M)$ with the subspace $\Sym_2(E^*)$ of $W$
is equal to ${\overline{Q_+}}.$ 
\end{pro}

\begin{proof} 
Let $\gamma$ be a proximal element of $\Gamma$ (see \ref{par:proximality}).
Then, $\gamma$ preserves a one dimensional eigenspace $R\subset \Sym_2(E^*)\subset W,$ with
the property that the eigenvalue $\lambda$ of $\gamma$ along $R$ 
dominates all other eigenvalues of $\gamma$ on $W\otimes_\R \C$ strictly. 
If  $u+v+t$ 
is an element of
$W= \Sym_2(E^*)\oplus E^k \oplus T^m,$ then 
$$
u_+ := \lim_{n\to +\infty} \frac{1}{\lambda^n} \,  \gamma^n(u+v+t)
$$
is contained in $R.$ 

Let us apply this remark to a generic element  $u+v+t$ of $\Ka(M).$ Since the K\"ahler cone is
open, we may assume that the vector $u_+\in R$ in the previous limit is different
from $0.$ The $\Gamma$-invariance of $\Ka(M)$ implies that $R\cap{\overline{\Ka}}(M)
\neq \{ 0\}.$ From this follows that the intersection between ${\overline{\Ka}}(M)$ 
and $\Sym_2(E^*)$ is a non trivial, $\Gamma$-invariant, strict, and closed convex cone
of $\Sym_2(E^*).$ It must therefore co\"{\i}ncide with $Q_+$ or its opposite. 
The conclusion follows from the normalization chosen in section \ref{par:CubicForm}
and the inequality
$$
\int_M u\wedge u \wedge u = \epsilon \, \det (u) < 0 
$$ 
if $u$ is in $-Q_+,$ which is not compatible with the existence of nef classes
in~$Q_+.$
\end{proof}

\section{Lattices in $\SL_3(\R)$: Part II}\label{chap:SL3II}

We pursue the proof of theorem A under the same assumptions
as in the previous section. While section \ref{chap:SL3I} focussed on the
cohomology of $M,$ we now use complex algebraic geometry
more deeply to conclude. 

\subsection{Contraction of the trivial factor}
\subsubsection{Ample, big and nef classes}
The following theorem describes the K\"ahler cone of any compact complex manifold
$M$ (see section \ref{par:conesdefi} for definitions).

\begin{thm}[Demailly, Paun \cite{Demailly-Paun:2004}]\label{thm:Demailly-Paun}
The K\"ahler cone $\Ka(M)$ of a (connected) compact K\"ahler manifold $M$  is a connected component of the cone $\PP(M)$ of numerically positive classes.
\end{thm}

As a consequence, if  the cohomology class $\alpha$ is on the boundary of $\Ka(M),$ there
exists an analytic subvariety $Y$ of $M$ such that $\int_Y \alpha^{\dim(Y)}=0$; 
in dimension $3,$ this leads to three cases: 
\begin{enumerate}
\item $\dim(Y)=3,$ i.e. $Y=M,$ and $\alpha^3= 0$;
\item $\dim(Y)=2,$ $Y$ is a surface and $\int_Y \alpha^2 = 0$;
\item $\dim(Y)=1,$ $Y$ is a curve and $\int_Y \alpha=0.$
\end{enumerate}

Another result that we shall use is the following characterisation of {\bf{big}}
line bundles. {\sl{Let $L$ be a line bundle on $M.$ If the first Chern class $c_1(L)$
is contained in the nef cone ${\overline{\Ka}}(M),$ 
and $\int_M c_1(L)^3 >0,$ then $L$ is big, in the sense that 
$$
\dim(H^0(M,L^{\otimes k})> c^{ste}\,  k^{\dim(M)}
$$
for $k>0$}} (see \cite{Demailly:LNM}, corollary 6.8, and \cite{Demailly-Paun:2004}).

\subsubsection{Base locus of $Q_+$}

The cone $Q_+$ is open, and  $H^2(M,\Z)$ intersects $\Sym_2(E^*)$ on a lattice. 
As a consequence, the set $Q_+\cap H^2(M,\Z)$ spans $\Sym_2(E^*)$ as a real vector space.
Let $u$ be an element of $Q_+\cap H^2(M,\Z).$ Lefschetz theorem on $(1,1)$-classes
implies that $u$ is the first Chern class of a line bundle $L$ on $M$
(see \cite{Voisin:book}, theorem 7.2, page 150). This line bundle is  big and 
nef because $Q_+$ is contained in ${\overline{\Ka}}(M)$ and 
$$
\int_M u^3 = \epsilon\, \det(u) >0.
$$ 
Let $\Bs(L)$ be the base locus of $L.$ If we choose two line bundles $L_1$ and $L_2$ with first
Chern class in $Q_+,$ the base locus  $\Bs(L_1\otimes L_2)$ is contained
in the intersection of $\Bs(L_1)$ and $\Bs(L_2).$ From this we deduce that
$$
\Bs(Q_+) := \bigcap_{L, \, c_1(L)\in Q_+} \Bs(L)
$$
is an analytic set, that coincides with $\Bs(\otimes_i L_i)$ if we take sufficiently
many line bundles $L_i$ with $c_1(L_i)\in Q_+.$ The set $\Bs(Q_+)$ will 
be refered to as the base locus of $Q_+.$ 

Since $Q_+$ is $\Gamma$-invariant, this base locus is a $\Gamma$-invariant 
analytic subset of $M.$ Theorem \ref{thm:class-inv-div} implies the following proposition,
because $\Gamma$ has property~(T). 

\begin{pro} The base locus of $Q_+$ is made of a fixed component, which 
is a disjoint union of projective planes, and of a finite number of points. Contracting
all $\Gamma$-periodic surfaces, one gets a birational 
morphism 
$
\pi:M \to M_0
$
onto an orbifold $M_0$ such that the base locus of $\pi_*Q_+$ is reduced to 
a finite set. 
\end{pro}

\subsubsection{Remarks concerning orbifolds}

Let $M_0$ be an orbifold of dimension~$3$; by definition, $M_0$ is locally isomorphic
to $\C^3/G_i$ where $G_i$ is a finite linear group.
A K\"ahler form on $M_0$ is a $(1,1)$-form $\omega$ which, locally
in a chart $U_i\subset M_0$ 
of type $\C^3/G_i$, lifts to a $G_i$-invariant
K\"ahler form $\tilde \omega$ in $\C^3.$ All other classical objects are defined
in a similar way (see \cite{Campana:2004} for more details): $(1,1)$-forms,
closed and exact classes, cohomology groups, etc.
For example, Chern classes can be defined  in terms of curvature
of hermitian metrics.  In a neighborhood
of a singularity, locally of type $\C^3/G_i$ where $G_i$ is a finite linear
group, the metric is supposed to lift to a $G_i$-invariant smooth metric on 
$\C^3$; the curvature is then invariant, and the Chern classes are well-defined
on $M_0.$

\subsubsection{Ampleness in $\pi_*Q_+$}
The morphism $\pi$ defines two linear operators $\pi^* : H^2(M_0,\R)\to H^2(M,\R)$
and $\pi_*:H_2(M,\R)\to H_2(M_0,\R).$ The first one is defined by tacking 
pull back of forms ; this is compatible with the operation of pull back for 
Cartier divisors and line bundles. The second one is defined by taking push forward
of currents ; this is compatible with the usual operator $\pi_*$ for divisors.
Since $\pi$ has degree $1,$ we have $\pi_*\circ \pi^*= {\text{Id}}$ on $H^2(M_0,\R).$  

Let us now study $\pi_*Q_+$ and show that this subspace of $H^{1,1}(M_0,\R)$ 
is indeed contained in the K\"ahler cone of $M_0.$ 

The morphism $\pi$ being $\Gamma$-equivariant, the linear map $\pi^*$ embeds
the representation of $G$ on $H^{1,1}(M_0,\R)$ into a subrepresentation of $W.$ On the other hand,  $\pi_*$ contracts a subfactor of the trivial part 
$T^m$ (corresponding to $\Gamma$-invariant surfaces), so
that $H^{1,1}(M_0,\R)$ is isomorphic, via $\pi^*,$ to a direct sum
$\Sym_2(E^*)\oplus E^k\oplus T^n,$ with $n\leq m.$
We shall prove that $n=0.$

The proof of theorem \ref{thm:Demailly-Paun}   can 
be adapted to the orbifold case almost word by word (see \cite{Campana:2004} 
for similar ideas), so that {\sl{if $V$ is a (connected) compact K\"ahler orbifold, the K\"ahler cone $\Ka(V)$ is
a connected component of $\PP(V).$}} 

\begin{cor}
The set $\pi_*Q_+$ is contained in the K\"ahler cone of $M_0.$ 
There are ample line bundles $L$ on $M_0$ with
$c_1(L)\in \pi_*Q_+.$
\end{cor}

\begin{proof}
To simplify notations, we shall denote the projection
$\pi_*Q_+$ by $Q_{+,0},$ the space $H^{1,1}(M_0,\R)$ by $W_0$ 
and the subspace $\pi_*\Sym_2(E^*)$ by $\Sym_2(E^*)_0,$ and use similar
notations for the dual vector space $W_0^*\subset H^{2,2}(M_0,\R).$

From proposition \ref{pro:Q+nef} we know that $Q_{+,0}$ is contained in the nef cone. 
Let us assume that $Q_{+,0}$ does not intersect the K\"ahler cone, so that $Q_{+,0}$ is in the boundary
of $\PP(M_0)$ (theorem \ref{thm:Demailly-Paun}). Let $v_1,$ ..., $v_6$ be elements of $Q_{+,0}$ that span $\Sym_2(E^*)_0$
as  a real vector space. The vectors $v_i\wedge v_j,$ $1\leq i \leq j\leq 6,$  span $\Sym_2(E)_0\subset W_0^*.$
Let $v$ be the sum of the $v_i$; $v$ is in $\partial \PP(M_0)$ and $v^3>0.$ Theorem \ref{thm:Demailly-Paun} provides
an analytic subset $Y\subset M_0$ 
of pure dimension $\dim(Y)=1$ or $2,$ such that
$$
\int_Y v^{\dim(Y)}=0.
$$ 
Since all $v_i$ are nef, this implies that $\int_Y v_i\wedge v_j=0$ for all $(i,j)$ if $\dim(Y)=2$ 
(resp. $\int_Y v_i =0,$ $\forall \, i,$ if $\dim(Y)=1$). In particular, $\int_Y w=0$ for every $w$ in $W_0^*$ 
(resp. $\int_Y w=0$ for all $w$ in $W_0$).

Let us now fix $Y$ as above, with dimension $2$ (resp. $1$), and 
let us decompose its cohomology class $[Y]\in W_0$ (resp. $W_0^*$) with respect to the
direct sum $W_0= \Sym_2(E^*)\oplus E^k\oplus T^m$ (resp. to its dual): 
$$
[Y]=a+b+c.
$$
Since 
$$
\int_Y w=0
$$
for all $w$ in $\Sym_2(E)_0$ (resp. $\Sym_2(E^*)_0$) we have $a=0.$ In particular, there is
an effective class $[Y]$ in $(E^*)^k  \oplus (T^*)^n$ that does not intersect $Q_+.$
Let ${\mathcal{C}}\subset  (E^*)^k\oplus (T^*)^n$ (resp. $\subset  E^k\oplus T^n$)  be the  cone
$$
{\mathcal{C}}=\{ w \in (E^*)^k\oplus (T^*)^n\vert\, w {\text{ is effective and}}\, \langle w \vert [Y]\rangle =0\}.
$$
Since this cone is strict,  lemma \ref{lem:coneE+T} and remark \ref{rem:lem5.10} imply
the existence of an effective divisor (resp. curve) $Z,$ with integer coefficients $a_i,$ 
$$
Z= \sum a_i Y_i,
$$
such that its homology class $[Z]$ is an element of the trivial factor $(T^*)^n$ (resp. $T^n$).

When the dimension of $Z$ is $2,$ we know that
$[Z]$ is contained in $T^n,$ so that the cohomology class of $Z$
is $\Gamma$-invariant, and corollary \ref{cor:det-hom-class} shows that $Z$ itself must
be $\Gamma$-invariant. But there is no $\Gamma$-invariant surface since
all of them have been contracted by $\pi.$ 

Let us now assume that $\dim(Z)=1.$ Let $Cov([Z])$ be the set
covered by all curves $Z'$ such that $[Z']=[Z].$ Since $[Z]$ is an element
of $T^n,$ its class is $\Gamma$-invariant. The set $Cov([Z])$
is therefore a $\Gamma$-invariant analytic set and, as such, must be all
of $M_0.$  Let $D_i,$ $1\leq i\leq k,$
be three divisors with cohomology classes in $Q_{+,0}$ such that the
intersection 
$$
\bigcap_{1\leq i\leq k} D_i
$$ 
is made of a finite number of 
points; such divisors exist because the base locus of $Q_{+,0}$ is reduced 
to a finite number of points. Let $Z'$ be a curve with $[Z']=[Z]$
which goes through at least one point of  $\cap_i D_i.$
Let $Z'_0$ be an irreducible component of $Z'$ containing that point.
Since $[Z'_0].[D_i]=0$ for all $i,$ the curve $Z'_0$ must be contained
in all of the $D_i,$ a contradiction. 
 
In both cases ($\dim(Z)=1$ or $2$) we get a contradiction. This implies
that $\pi_*Q_+$ intersects the K\"ahler cone of the orbifold $M_0,$ 
and proposition \ref{pro:coneQ+carac} shows that $\pi_*(Q_+)$ is contained in $\Ka(M_0).$ Since 
$Q_+$ intersects $H^2(M,\Z),$ there are ample line bundles $L$ on $M_0$ 
with $c_1(L)$ in $\pi_*(Q_+).$ 
\end{proof}

\subsubsection{Contraction of $\,T^m$}

We shall now prove that $n=0,$ i.e. that $\pi_*$ contracts all classes $t$ in $T^m.$
This is a consequence of  Hodge index theorem.

\begin{cor}\label{cor:noinvclass}
The cohomology classes of all $\Gamma$-periodic surfaces $S_i\subset M$ span
$T^m.$ Since all of them are contracted by $\pi:M\to M_0,$ the trivial
factor $T^m$ is mapped to $0$ by $\pi_*,$ and there is no non-zero 
 $\Gamma$-invariant cohomology class on $M_0.$ \end{cor}

\begin{proof}
Let $u$ be a K\"ahler class in $\pi_*Q_+.$ Then $u\wedge u$ is different from $0,$
and $E^k \oplus T^n $
is contained in $(u\wedge u) ^\perp.$ Since $u$ is in the K\"ahler cone, Hodge index
theorem implies that the quadratic
form 
$$
v\mapsto - \int_M u\wedge v\wedge v
$$ 
is positive definite on $(u\wedge u)^\perp.$ But from section \ref{par:sectionwedge} we know that all elements 
$t\in T^m$ satisfy 
$
 \int_M u\wedge t\wedge t = 0.
$
All together,  $n$ is equal to $0.$  
\end{proof}

\begin{rem}
We could also apply the hard Lefschetz's theorem, as in the proof of lemma 
\ref{lem:DonT}, to prove
this corollary (see \cite{Zaffran-Wang:2009} for a proof of the hard Lefschetz theorem in the orbifold case). 
\end{rem}

\subsection{Triviality of Chern classes and tori}$\,$

\vspace{0.16cm}

Let $c_1(M_0)$ and $c_2(M_0)$ be the Chern classes of the orbifold
$M_0.$ These classes are respectively of type $(1,1)$ and 
$(2,2),$ and are invariant under the 
action of $\Gamma$ on  the cohomology of $M_0.$ From corollary
\ref{cor:noinvclass}, any $\Gamma$-invariant class is equal to $0,$ proving that {\sl{the
Chern classes of the orbifold $M_0$ are equal to $0$}}. 
We now use the following result, the proof of which is described in
\cite{Kobayashi:book}, and is easily adapted to the orbifold case (see \cite{Campana:2004} for similar extensions of classical results to the orbifold setting). 

\begin{thm}
Let $X_0$ be a connected K\"ahler orbifold with trivial first and second Chern classes. 
Then $X_0$ is covered by a torus: There is an unramified covering 
$\epsilon:A\to X_0$ where $A$ is a torus of dimension $\dim(X_0).$ 
\end{thm}

Let us be more precise. In our case, $M_0$ is a three dimensional (connected)
orbifold with trivial Chern classes $c_1(M_0)$ and $c_2(M_0).$ This implies that there 
is a flat K\"ahler metric on $M_0$ (see \cite{Kobayashi:book}).
The universal cover of $M_0$ (in the orbifold sense) is then isomorphic to 
$\C^3$  and the (orbifold) fundamental group $\pi_1^{orb}(M_0)$ acts 
by affine isometries on $\C^3$ for the standard euclidean metric. 
In other words, $\pi_1^{orb}(M_0)$
 is identified to a cristallographic group $\Delta$ of affine motions of $\C^3.$ 
 Let $\Delta^*$ be the group of translations contained in $\Delta.$ 
 Bieberbach's theorem shows that  (see \cite{Wolf:book}, chapter
 3, theorem 3.2.9).  
 \begin{itemize}
\item[a.-] $\Delta^*$ is a lattice in $\C^3$;
\item[b.-] $\Delta^*$ is the unique maximal and normal free abelian subgroup 
of $\Delta$ of rank $6.$
 \end{itemize}
The torus $A$ in the previous theorem is the quotient  of $\C^3$ 
by this group of translations. By construction, $A$ covers $M_0,$
 Let $F$ be the quotient  group $\Delta/\Delta^*$; we identify it to the 
 group of deck transformations of the covering $\epsilon:A\to M_0.$

\begin{lem}\label{lem:Fcyclic} $\,$
\begin{enumerate}
\item A finite index subgroup of $\Gamma$ lifts 
to $\Aut(A).$
\item Either $M_0$ is singular, or $M_0$ is a torus. 
\item If $M_0$ is singular, then $M_0$ is a quotient of the torus 
$A$ by a homothety $(x,y,z)\mapsto (\eta x, \eta y, \eta z),$ where $\eta$
is a root of $1.$
\end{enumerate}
\end{lem}

\begin{proof}
By property (b.) all automorphisms of $M_0$ lift to $A.$ Let $\overline{\Gamma} \subset \Aut(A)$ be the group of automorphisms of $A$ made of all possible lifts of elements
of~$\Gamma.$ So, 
$\overline{\Gamma}$ is an extension of  $\Gamma$ by $F$: 
$$
1\to F\to \overline{\Gamma} \to \Gamma \to 1.$$
Let $L:\Aut(A)\to \GL_3(\C)$ be the morphism which applies any automorphism
$f$ of $A$ to its linear part $L(f).$ Since $A$ is obtained as the quotient of $\C^3$ 
by all translations of $\Delta,$ the restriction of $L$ to $F$ is injective. 
Let $N\subset \GL_3(\C)$ be the normalizer of $L(F),$ and $N^0$ the 
connected component of the identity of $N.$  
The group $L( \overline{\Gamma})$ normalizes $L(F).$ Hence we have a well defined
morphism $\overline{\Gamma}\to N,$ and an induced morphism $\delta:\Gamma\to N/L(F).$ 
Changing $\Gamma$ into a finite index subgroup, we may and do assume that
$\delta$ is injective and $\delta(\Gamma)$ is contained in $N^0/L(F).$
Let us now assume that $L(F)$ is not contained in the center of $\GL_3(\C).$ 
Then there is an element $g$ of $F$ which is not an homothety; in particular, 
$g$ has an eigenspace $V$ of dimension $1$ or $2.$ Since this space is $N^0$ invariant, $N^0$ embeds in the stabilizer of $V.$ But $V$ and $\C^3/V$ both have dimension
at most $2.$ Hence, if $H$ is a discrete group with Kazhdan property (T), all morphisms
from $H$ to $N^0$ (resp. to $N^0/L(F)$) have finite image (see \S \ref{par:PropT}). This contradicts the
existence of $\delta:\Gamma\to N^0/L(F).$ As a consequence, $L(F)$ is contained in the
center of $\GL_3(\C).$ 

Either $F$ is trivial, and then $M_0$ coincides with the torus $A,$ or $F$ is
a cyclic subgroup of $\C^* {\text{Id}}.$ In the first case, there is no need to 
lift $\Gamma$ to $\Aut(A).$
In the second case, we fix a generator $g$ 
of $F,$ and denote by $\eta$ the root of unity such that $L(g)$ is multiplication
by $\eta.$ The automorphism $g$ has at least one (isolated) fixed point $x_0$ 
in $A$; this point is fixed by $F.$ Changing $\Gamma$ into a finite
index subgroup $\Gamma_1,$ we can also assume that $\overline{\Gamma_1}$ 
fixes $x_0.$ The linear part $L$ embeds $\overline{\Gamma_1}$ 
into $\GL_3(\C).$ Selberg's lemma assures that a finite index subgroup
of $\overline{\Gamma_1}$ has no torsion. This subgroup does not intersect
$F,$ hence projects bijectively onto a finite index subgroup of $\Gamma_1.$ 
This proves that a finite index subgroup of $\Gamma$ lifts to $\Aut(A).$
\end{proof}
 
\subsection{Possible lattices and tori}$\,$

\vspace{0.16cm}

We now show that $\Gamma$ is commensurable to $\SL_3(\Z)$ and
that $A$ is isogenous to a product $B\times B \times B$ where $B$ 
is an elliptic curve. For this purpose, we shall use the following proposition,
and refer to the next section for a proof. 

\begin{pro}\label{pro:latticesl3Z}
Let $\Gamma$ be a lattice in $\SL_3(\R),$ which preserves a lattice 
${\mathcal{L}}$ in $\Sym_2(E^*).$ If there exists a real number $\epsilon$ such that $\det:\Sym_2(E^*)\to \R$ maps  ${\mathcal{L}}$ into $\epsilon  \Q,$
then $\Gamma$ is commensurable with $\SL_3(\Z).$
\end{pro}

This proposition, proposition \ref{pro:invlattices} and lemma \ref{lem:cubicform} show
that $\Gamma$ is commensurable with $\SL_3(\Z).$ After a conjugacy in $\SL_3(\R),$
we may assume that $\Gamma$ intersects $\SL_3(\Z)$ on a finite index subgroup.
As a consequence, there exists a finite index subgroup $F$ of $\SL_2(\Z)$
such that 
$$
\left(   
\begin{array}{ccc}
a & b & 0 \\
c & d & 0 \\
0 & 0 & 1
\end{array}
\right) \in \Gamma, \quad \forall \left(   
\begin{array}{cc}
a & b \\
c & d  \\
\end{array}
\right) \in F.
$$ 
The group $F,$ viewed as a subgroup of $\Gamma,$
acts on the subspace $\Sym_2(E^*)$ of $W= H^{1,1}(A,\R),$
preserving the lattice ${\mathcal{R}}=\Sym_2(E^*)\cap H^2(A,\Z).$ This lattice
defines a $\Q$-structure on the real vector space $\Sym_2(E^*).$
Since $F$ preserves ${\mathcal{R}},$ the intersection of all subspaces
$$
V_1(\gamma)=\{v\in \Sym_2(E^*); \, \gamma(v)=v\}
$$
where $\gamma$ describes $F$ is defined over $\Q.$ 
This space has dimension one, and is generated by the cohomology
class of a $(1,1)$-form of rank $1.$ This implies
that there exists a line bundle $L$ on $A$ with a first Chern class
$c_1(L)$ in $\Sym_2(E^*)$ that is invariant under the action of $F.$ 
Moreover, this Chern class being proportional to a
rank $1$ hermitian form, $L$ determines
a morphism to an elliptic curve $B$:
$$
\phi_L:A\to B.
$$ 
Composing $\phi_L$ with elements in $\Gamma$ which are not
in $F,$ we can construct two new projections $\Phi_L\circ\gamma_i:A\to B$
such that the holomorphic $1$-forms $D\Phi_L,$ $D\Phi_L\circ \gamma_1,$ 
and $D\Phi_L\circ \gamma_2$ generate $T^*A$ at every point. It
follows that the map
$$
\Phi_L\times (\Phi_L\circ\gamma_1)\times (\Phi_L\circ\gamma_2): A\to B\times B\times B
$$
is an isogeny. 
This proves that $A$ is isogenous to a product $B\times B\times B.$

\subsection{Proof of proposition \ref{pro:latticesl3Z}}$\,$

\vspace{0.16cm}

Let us prove proposition \ref{pro:latticesl3Z}. By assumption, ${\mathcal{L}}$ is a lattice
in $\Sym_2(E^*),$ and $\Gamma$ is a lattice in $\SL_3(\R)$ 
which preserves ${\mathcal{L}}.$ The lattice ${\mathcal{L}}$ determines a $\Q$-structure
on the real vector space $\Sym_2(E^*),$ and $\Gamma$ acts by matrices with
integer coefficients with respect to this structure.

From the classification of lattices in $\SL_3(\R),$ we know that
$\Gamma$ is commensurable to $\SL_3(\Z),$ $\Gamma$ has
$\Q$-rank $1,$ or $\Gamma$ is cocompact (see \cite{Witte:prebook}). What we have to do, 
is to rule out the last two possibilities.

\subsubsection{Cocompact lattices}

Let us assume that $\Gamma$ is cocompact, and get a contradiction.

Let $q_0$ be an element of ${\mathcal{L}}$ 
which, as a quadratic form on $E,$ has signature $(+,+,-).$ Then there
exists a basis of $E$ such that 
$$
q_0(x,y,z)=2xy+z^2
$$
(note that the value of $\epsilon$ may change after the choice of
a new basis for $E$). 
The stabilizer $H$ of $q_0$ in $\SL_3(\R)$ is isomorphic to the group 
$\SO_{2,1}(\R).$ The orbit of $q_0$ under $\SL_3(\R)$ may be
identified to $\SL_3(\R)/H.$ The orbit of $q_0$ under $\Gamma$ is contained
in ${\mathcal{L}},$ and is therefore discrete in $\SL_3(\R)/H.$ From
this we deduce that the $H$-orbit of $q_0$ in 
$\Gamma\backslash \SL_3(\R)$ is closed. 
Since $\Gamma\backslash \SL_3(\R)$ is compact, the $H$-orbit must
be compact too, so that $\Gamma\cap H$ is a cocompact lattice in 
$H.$ 

As a consequence, we can find an element $\gamma$ in $\Gamma\cap H$
which is a proximal element of $H$; this means that $\gamma$ has $3$ distinct eigenvalues $\lambda,$
$1/\lambda,$ and $1,$ with $\vert \lambda \vert>1.$ After conjugation inside $H,$ 
we can assume that the set of fixed points of the endomorphism
$$
\gamma:\Sym_2(E^*)\to \Sym_2(E^*)
$$ 
coincides with the  plane 
$$
V_1= {\text{Vect}}\{q_1,q_2\}
$$ 
where $q_1$ and $q_2$ are the quadratic forms $q_1=z^2$ and $q_2=2xy.$
Since $\gamma$ preserves ${\mathcal{L}},$ the plane $V_1$ is defined over $\Q$
(with respect to the $\Q$-structure given by ${\mathcal{L}}$). In particular, 
${\mathcal{L}}\cap V_1$ is a (cocompact) lattice in $V_1.$

Let $r_0=2\alpha xy + \beta z^2$ be an element of ${\mathcal{L}}\cap V_1$ which is not proportional to $q_0,$ so that $\alpha\neq \beta.$ 
Computing $\det(mq_0+nr_0)$  we see that
$$
(m+n\alpha)^2(m+n\beta)\in \epsilon\Q
$$ 
for all integers $m$ and $n.$ With $n=0,$ we see that $\epsilon$ is rational.
With $m=0,$ we see that $\alpha^2\beta$ is a
rational number. This implies that the affine function
$$
(m,n) \mapsto \beta + 2\alpha m + (2\alpha \beta +\alpha^2)n
$$
takes rational values for $(m,n)\in \Z^2,$ $mn\neq 0.$ As a consequence, 
both $\alpha$ and $\beta$ are rational numbers.

From this we deduce the existence of positive integers $a$ and $b$ such that $aq_1$
and $bq_2$ are contained in ${\mathcal{L}}\cap V_1.$ In particular, ${\mathcal{L}}$ 
contains a multiple of the rank $1$ quadratic form $q_1.$ The cone
of quadratic forms of rank $1$   is parametrized by the
map 
$$
s:\left\{\begin{array}{rcl} E^*\setminus\{0\} & \to & \Sym_2(E^*) \\
f& \mapsto & s(f)=f^2 \end{array}\right.
$$
This map is a $2$-to-$1$ cover and, in particular, is proper. 
Since $s$ is $\Gamma$-equivariant (i.e. $s(\gamma.f)=\gamma(s(f))$), 
the $\Gamma$-orbit of $s^{-1}(q_1)$ is a discrete subset of $E^*.$ 
This contradicts the following theorem, due to Greenberg  
(see \cite{Auslander-Green-Hahn:book}, appendix): {\sl{
Let $\Gamma$ be a cocompact lattice in $\SL_n(\R)$ (resp. $\SL_n(\C)$); 
if $v$ if an element of $\R^n\setminus \{0\}$ 
(resp. $\C^n\setminus \{0\}$), the orbit $\Gamma v$ is a dense
subset of $\R^n$ (resp $\C^n$).}}

\subsubsection{$\Q$-rank-$1$ lattices}

Let us now assume that the $\Q$-rank of $\Gamma$ is equal to~$1.$ 
In this case, $\Gamma$ is obtained from the construction in 
chapter 7.E in the book \cite{Witte:prebook}. In particular, after conjugacy,
$\Gamma$
contains matrices of type 
$$
\left(   
\begin{array}{ccc}
1 & a+b\sqrt{r} & -(a^2-rb^2)/2 + c\sqrt{r} \\
0 & 1 & -a + b\sqrt{r} \\
0 & 0 & 1
\end{array}
\right) 
$$
where $r$ is a fixed square free integer, and 
$a,$ $b,$ and $c$ are  sufficiently divisible integers.
The group $\Gamma$ contains also the transpose
of those matrices. 

For the representation $\rho:\SL_3(\R)\to \GL(\Sym_2(E^*)),$  
we have 
$$
\tr(\rho(B^{-1}))= \tr(B^2) - \tr(B^{-1})
$$
for all $B\in\SL_3(\R).$ 
Let us apply this formula to $B= \, ^t\! \!A A'$ where $A$ and $A'\in\Gamma$ 
are as above. A straightforward computation shows that
the trace is not rational in general. This implies that the action of 
$\Gamma$ on $\Sym_2(E^*)$ can not preserve a lattice ${\mathcal{L}}.$

%
%
\section{Lattices in $\SL_3(\C)$}\label{chap:SL3C}
%
%

We now study holomorphic actions of lattices in a real Lie group $G$ which 
is locally isomorphic to $\SL_3(\C).$ 
As before, $\Gamma$ will be a lattice in $G$ that acts faithfully on a compact
K\"ahler threefold $M,$ and we assume that the action of $\Gamma$ on 
$W=H^{1,1}(M,\R)$ extends to a non trivial linear representation 
$$
G\to \GL(W).
$$
Our goal is to prove that 
\begin{enumerate}
\item the action of $\Gamma$ on $M$ is virtually of Kummer type (coming from 
an action of a finite index subgroup $\Gamma_0$ in $\Gamma$ on a 
three dimensional compact torus $A$);
\item $\Gamma$ is commensurable to $\SL_3({\mathcal{O}}_d)$ 
where ${\mathcal{O}}_d$ is the ring of integers in a quadratic
field $\Q(\sqrt{d})$ for some negative squarefree integer $d$;
\item the torus $A$ is isogenous to a product $B\times B\times B$ where
$B$ is the elliptic curve $\C/{\mathcal{O}}_d.$ 
\end{enumerate}
Since the proof follows the same lines as for lattices in $\SL_3(\R),$ 
we just list the results that are required to adapt sections \ref{chap:SL3I} and  \ref{chap:SL3II}.

\subsection{The linear representation on $W$}$\,$

\vspace{0.16cm}

From theorem \ref{thm:sl3C}, we can assume that 
there exist three integers $k\geq 0,$ $l\geq 0,$ and $m\geq 0$ such that the representation $W$ 
of $\g= \sll_3(\C)$ is isomorphic to 
$$
W=\Her(E_\C^*))\oplus E_\C^k\oplus{\overline{E_\C}}^l \oplus T^m
$$
after composition by an automorphism of $\g.$

\subsection{The cup product and the integer structure}$\,$

\vspace{0.16cm}

As in section \ref{par:sectionwedge}, one then proves that the cup product satisfies properties (1) to (5) (\S \ref{par:sectionwedge}) where $\Sym_2(E^*)$ must be replaced by $\Her(E_\C^*)$ and
$E^k$ by $E_\C^k\oplus E_\C^l.$ 

Once again, the lattice $H^2(M,\Z)$ intersects the irreducible factor $\Her(E_\C^*)$
on a (cocompact) lattice.

\subsection{Invariant cones}$\,$

\vspace{0.16cm}

The cone $Q_+ \subset \Sym_2(E^*)$ must now be replaced by the cone
$H_+ \subset \Her(E_\C^*)$ of positive definite hermitian forms. This
cone admits the following characterization.

\begin{lem}
If ${\mathcal{C}}$ is a $\Gamma$-invariant, strict, and open convex cone in
$\Her(E_\C^*),$ ${\mathcal{C}}$ co\"{\i}ncides with $H_+$ or its opposite
$-H_+.$ If ${\mathcal{C}}$ is a $\Gamma$-invariant, strict, and closed convex cone,
then ${\mathcal{C}}$ is either $\{ 0 \},$ $\overline H_+,$ or $- \overline H_+.$ 
\end{lem}

\begin{proof}[Sketch of the proof]
Every orbit of the lattice $\Gamma$ in $\P(E)$ is dense, so that every orbit
of $\Gamma$ on the cone of rank one positive hermitian forms $\xi \otimes {\overline{\xi}}$ 
is dense. As a consequence, the projectivized  boundary of any strict invariant cone ${\mathcal{C}}$
must contain the set of rank one hermitian forms. Since the cone of rank one positive hermitian
forms contains $H_+$ in its convex hull, the lemma follows.
\end{proof}

Similarly, one shows that there is no invariant convex cone  ${\mathcal{C}}$ in $E_\C$  (resp. ${\overline{E_\C}}$)
except the trivial ones $\{0\}$ and $E_\C$ (resp. ${\overline{E_\C}}$). This implies that any non trivial, strict, closed,
and $\Gamma$-invariant convex cone ${\mathcal{C}}$ in $E_\C^k\oplus E_\C^l\oplus T^m$ intersects
$T^m\setminus\{0\}.$

\subsection{Conclusion}$\,$

\vspace{0.16cm}

The strategy used in section \ref{chap:SL3II} can now be applied word by word. 
It shows that the action $\Gamma\times M\to M$ is virtually of Kummer
type. Changing $\Gamma$ in one of its finite index subgroups, this action 
comes from a linear action of $\Gamma$ on a compact complex
torus $A.$ 

\begin{pro}
Let $\Gamma$ be a lattice in $\SL_n(\C).$ If there exists a $\Gamma$-invariant
lattice ${\mathcal{L}}$ in $\C^n$ then $\Gamma$ is commensurable to $\SL_n({\mathcal{O}}_{d})$
for some square free negative integer $d.$ 
\end{pro}

\begin{proof}
Let $G$ be the group $\SL_n(\C).$ This is an algebraic subgroup of $\SL_{2n}(\R)$
acting linearly on $\R^{2n}=\C^n$ and preserving the complex structure. By 
assumption, the lattice
$\Gamma$ in $G$ preserves a lattice ${\mathcal{L}}\subset \R^{2n}.$  

Let $(e_1,..., e_n)$ be a family of $n$ linearly independant vectors $e_i\in {\mathcal{L}}$ such that 
$$
P:={\text{Vect}}_\R(e_1,...,e_n)
$$
is a totally real subspace in $\C^n,$ i.e. $P\oplus {\sqrt{-1}}P=\C^n.$ Then $(e_1, ..., e_n)$
is a basis of $\C^n$ as a complex vector space. Changing ${\mathcal{L}}$
into one of its finite index lattices, we may assume that $e_1,$ $e_2,$ ..., and $e_n$
are the first elements in a basis $(e_1, ..., e_{2n})$ of ${\mathcal{L}}$ as a $\Z$-module.
In this new basis (i.e. after conjugation by an element $B$ in $\GL_{2n}(\R)$), 
${\mathcal{L}}$ is the standard lattice $\Z^{2n},$  $\Gamma $ is a subgroup
of $\SL_{2n}(\Z),$ and since $G$ contains $\Gamma$ as a lattice, $G$ is defined
over $\Q.$ Borel-Harish-Chandra theorem (see section \ref{par:BHC})
shows that $\Gamma$ has finite index in the lattice $G\cap \SL_{2n}(\Z).$

In this new basis, the $n$-plane $P$ is obviously defined
over $\Q,$ and Borel-Harish Chandra theorem shows that
$SL_{2n}(\Z)$ intersects the stabilizer of
$P$ in $G$ onto a lattice $\Gamma'.$ Let us use
$(e_1,..., e_n)$ as a complex basis of $\C^n.$
The stabilizer of $P$ coincides with $\SL_n(\R)\subset \SL_n(\C),$
and $\Gamma' $ has finite index in $\SL_n(\Z)\subset\SL_n(\R).$

Let $\pi$ be the projection of $\R^{2n}$ onto ${\sqrt{-1}}P$ parallel to $P.$ The projection 
of ${\mathcal{L}}$ is then a lattice $\pi({\mathcal{L}})$ in ${\sqrt{-1}}P.$ 
The group $\Gamma',$ viewed as a subgroup
of $\SL_n(\C)$ in the basis $(e_1,..., e_n),$ is given by matrices with integer
coefficients. As such, it preserves both $P$ and ${\sqrt{-1}}P.$ 
On the real vector space ${\sqrt{-1}}P,$ we use the basis
$({\sqrt{-1}}e_1, ..., {\sqrt{-1}}e_n),$ and choose a matrix $M\in \GL({\sqrt{-1}}P)=\GL_n(\R)$ such that
$$
\pi(L) = M(\Z {\sqrt{-1}}e_1 \oplus ... \oplus \Z {\sqrt{-1}}e_n).
$$
Since $\Gamma'$ preserves $\pi({\mathcal{L}}),$ both $M\Gamma' M^{-1}$ and $\Gamma'$
are finite index subgroups of $\SL_n(\Z).$ This implies that $\pi({\mathcal{L}})$
is, up to finite indices, equal to ${\sqrt{-1}}\delta(\Z e_1 \oplus ... \oplus \Z e_n)$ for 
some positive real number $\delta.$
As a consequence, up to finite indices, we may assume that ${\mathcal{L}}$ coincides 
with the subgroup 
$$
\Z^n\oplus {\sqrt{-1}}\delta \Z^n
$$
in the basis $(e_1, ..., e_n)$ of $\C^n.$ 

Let $\Gamma_1$ be the finite index subgroup of $\Gamma$ that preserves
$\Z^n\oplus {\sqrt{-1}}\delta \Z^n.$
Let $A+{\sqrt{-1}}B$ be an element of $\Gamma_1$ where $A$ and $B$
are $n\times n$ real matrices. Then 
\begin{eqnarray*}
Ax-\delta By \in \Z^n  \text{ and }
\delta A y +  B x \in \delta \Z^n \, {\text{ for all }} \,(x,y)\in \Z^n\times \Z^n.
\end{eqnarray*}
Hence,  $A$ is an element of $\Mat_n(\Z),$ $B$ is an
element of $\delta \Mat_n(\Z),$ and $\delta^2$ is a positive integer.
Let $d$ be the negative integer $-\delta^2.$  Then 
${\mathcal{L}}$ coincides with $\Z^n\oplus {\sqrt{d}}\Z^n$
and $\Gamma$ with $\SL_3({\mathcal{O}}_{ d}),$  up to finite indices.
\end{proof}

Let us now apply this proposition to our context.
Since we can change $\Gamma$ into finite index subgroups and $A$ 
into isogenous tori, we may assume that $A$ is the quotient of $\C^3$ 
by the lattice $\Z^3\oplus {\sqrt{d}}\Z^3,$ and $\Gamma$ has finite index
in the lattice $\SL_3({\mathcal{O}}_{d}).$ In particular, $A$ is (isogenous to)
$B\times B \times B$ where $B= \C/{\mathcal{O}}_{d}.$
This proves theorem A when $G=\SL_3(\C),$ which was the last remaining
case.

%
%
\section{Complements}
%
%

\subsection{Kummer examples}\label{par:Kummer}$\,$

\vspace{0.16cm}

Let $A$ be a torus of dimension $3$ with a faithful action of an irreducible
higher rank lattice $\Gamma.$ We proved that $A$ is isogenous 
to a product $B\times B\times B$ and that $\Gamma$ contains
a subgroup $\Gamma_0$ which is commensurable to $\SL_3(\Z)$
(note that $\SL_3({\mathcal{O}}_d)$ contains  $\SL_3(\Z)$).

Let now $F$ be a finite subgroup of $\Aut(A).$ If the orbits of 
$F$ are permuted by the action of $\Gamma,$ then $\Gamma$ 
normalizes $F$:
$
\gamma F \gamma^{-1}=F
$ 
for all $\gamma$ in $\Gamma.$ Changing $\Gamma_0$ into 
a finite index subgroup, we can and do assume that $\Gamma_0$ 
commutes with all elements of $F.$ 
Let $F_0$ be the group of translations contained in $F$: 
$$
F_0= \Aut(A)^0\cap F.
$$
This group is normal, and commutes to $\Gamma_0.$ Changing $A$
into the torus $A'=A/F_0,$ we can assume that $F_0$ is trivial or, equivalently, 
that the morphism 
$$
f\mapsto L(f)
$$ 
which maps an automorphism onto its linear part, is injective. 
Under these assumption, we proved in lemma \ref{lem:Fcyclic} that 
$F$ acts as a finite cyclic group of homotheties: $F$ is
generated by $(x,y,z)\mapsto (\eta x, \eta y, \eta z)$
where $\eta$ is a root of $1$ (for this we put
the origin of $A$ at a fixed point of $F$).
Multiplication by $\eta$ must preserve a lattice of type $\Lambda \times \Lambda \times \Lambda$ where $B=\C/\Lambda$ is an elliptic curve, and thus preserves
a lattice $\Lambda \subset \C.$ 
As a consequence, $\eta$ is equal to $-1,$ $i=\sqrt{-1},$ $e^{i\pi/3},$ or to one 
of their powers (to prove it, remark that multiplication by $\eta$ is a finite
order element of $\SL(\Lambda)\simeq \SL_2(\Z)$).

\begin{pro}
Let $M_0$ be a Kummer orbifold $A/F$ where $A$ is a torus of dimension 
$3$ and $F$ is a finite subgroup of $\Aut(A).$ Assume that 
\begin{itemize}
\item[(i)] $M_0$ is not a torus, and
\item[(ii)] $F$ is normalized by an irreducible higher rank lattice $\Gamma\subset \Aut(A).$
\end{itemize}
Then $M_0$ is isomorphic to a quotient $A'/F'$ where $A'$ is a torus
and $F'$ is a cyclic 
subgroup of $\Aut(A')$ generated by a homothety
$
f(x,y,z) = (\eta x, \eta y, \eta z),
$
where $\eta$ is a root of $1$ of order $1,$ $2,$ $3,$ $4$ or~$6.$ 
\end{pro}

\subsection{Volume forms}\label{par:VolumeForm}$\,$

\vspace{0.16cm}

Let us start with an example. Let $M_0$ be the Kummer orbifold
$A/F$ where $A$ is $(\C/\Z[i])^3$ and $F$ is the finite cyclic group
generated by 
$$
f(x,y,z)= (ix,iy,iz),
$$
where  $i={\sqrt{-1}}.$ Let $M$ be the smooth
manifold obtained by blowing up the singular points of $M_0.$ 
Let 
$$
\Omega=dx\wedge dy \wedge dz
$$
be the standard volume form on $A.$ Then $f^*\Omega = i\Omega$
and $\Omega^4$ is $f$-invariant ($\Omega^4$ may be viewed as
a section of $K_A^{\otimes 4},$ where $K_A=\det(T^*A)$ is the canonical bundle
of $A$). In order to resolve the singularities of $M_0,$ one can proceed as follows. 
First one blows up all fixed points of elements  in $F\setminus\{{\text{Id}}\}.$ For example, 
one needs to blow up the origin $(0,0,0).$ This provides a compact 
K\"ahler manifold ${\hat{A}}$ together with a birational morphism $\alpha :{\hat{A}}\to A.$
The automorphism $f$ lifts to an automorphism $\hat{f}$ of ${\hat{A}}$; since
the differential $Df$ is a homothety, $f$ acts trivially on each exceptional divisor, 
and acts as $z\mapsto iz$ in the normal direction. As a consequence, the 
quotient ${\hat{A}}/{\hat{F}}$ is smooth. 

Let us denote by $E\subset {\hat{A}}$  the exceptional divisor  corresponding to the blowing up of the origin, and fix local coordinates $({\hat{x}}, {\hat{y}}, {\hat{z}})$ in ${\hat{A}}$ such that the local equation of $E$ is ${\hat{z}}=0.$ In these coordintates, 
the form $\alpha^*\Omega$ is locally 
given by 
$$
\alpha^*\Omega =   {\hat{z}}^2  {\hat{x}} \wedge {\hat{y}} \wedge {\hat{z}}.
$$
The projection $\epsilon: {\hat{A}}\to M = {\hat{A}}/{\hat{F}}$ is given by
$
({\hat{x}}, {\hat{y}}, {\hat{z}}) \mapsto (u,v,w)=({\hat{x}}, {\hat{y}}, {\hat{z}}^4),
$
and the projection of $\alpha^*\Omega$  on $M$ is 
$$
\epsilon_*\alpha^*\Omega= \frac{1}{4w^{1/4}}du \wedge dv\wedge dw.
$$
This form is locally integrable, and its fourth power is a well defined meromorphic
section of $K_M^{\otimes 4}.$ 

A similar study can be made for all
Kummer examples. More precisely, a local computation shows that, after one 
blow up, the volume form $\epsilon_* \alpha^* \Omega$ 
on the quotient  ${\hat{\C^3}}/\eta$ 
\begin{itemize}
\item vanishes along the exceptional divisor $\epsilon(E)$ with multiplicity $1/2,$
if $\eta=-1,$
\item has a pole of type $1/w^{1/4}$ if $\eta$ has order $4,$
\item is smooth and does not vanish if $\eta$ has order $3,$ 
\item has a pole of type $1/w^{1/2}$ if $\eta$ has order $6.$
\end{itemize}
As a consequence,  {\sl{
the real volume form  $\epsilon_* \alpha^* (\Omega\wedge  {\overline{\Omega}})$ on $M$ is integrable and $\Gamma$-invariant}} (this form is not smooth if $\eta$ is not in $\{ 1, -1, e^{2i\pi/3}, e^{2i\pi/3}\}$). 

\begin{cor}
Let $M$ be a compact K\"ahler manifold of dimension $3.$ 
Let $\Gamma$ be a lattice in a simple Lie group $G$
with $\rk_\R(G)\geq 2.$
If $\Gamma$ acts faithfully on $M,$ then the action of $\Gamma$ on $M$ 
\begin{itemize}
\item virtually extends to 
an action of $G,$ or
\item preserves an integrable volume form $\mu$ which is locally  smooth or
the product of a smooth volume form by $\vert w \vert ^{-1/N},$ 
where $w$ is a local coordinate and $N=1$ or $2.$ \end{itemize}
\end{cor}

\begin{proof}
If the action of $\Gamma$ on the cohomology of $M$ factors through a
finite group, then $\Gamma$ is
virtually contained in $\Aut(M)^0$ and two cases may occur. In the first case, the morphism $\Gamma\to \Aut(M)$
virtually extends to a morphism $G\to \Aut(M).$
In the second case, $\Gamma$ is virtually contained in a  
compact subgroup of $\Aut(M)^0,$ and then $\Gamma$ preserves
a K\"ahler metric. In particular, it preserves a smooth volume form.

If the action of $\Gamma$ on the cohomology is almost faithful, then it is a Kummer example,
and the result follows from what has just been said.
\end{proof}

\subsection{Calabi-Yau examples}$\,$

\vspace{0.16cm}

Let $M$ be a compact K\"ahler manifold of dimension $3.$ By definition, 
$M$ is a (irreducible) Calabi-Yau manifold if its fundamental group is finite
and its first Chern class is  trivial. 
Another definition, which is not equivalent to the previous
one, requires a trivial first Chern class and a trivial first Betti number. 
The difference between the two definitions comes from the existence of 
smooth quotients of tori $A/F$ with trivial first Betti number 
(the fundamental group has a finite abelianization, see \cite{Oguiso-Sakurai:2001}). Both definitions work 
for the following corollary.

\begin{cor}
Let $M$ be a Calabi-Yau manifold of dimension $3.$
If $\Aut(M)$ contains an irreducible higher rank lattice, 
then $M$ is birational to the quotient of $(\C/\Z[j])^3$ by multiplication by $j,$
where $j=e^{2i\pi/3}.$
\end{cor}

\begin{proof}
Let $M$ be a Calabi-Yau manifold of dimension $3$ such that $\Aut(M)$
contains an irreducible higher rank lattice $\Gamma.$ Since $\Aut(M)$ is discrete, 
we can and do assume that $\Gamma$ acts faithfully on the cohomology of $M.$ 
Contracting all $\Gamma$ periodic surfaces as in section \ref{par:cis}, theorem 
A shows that we get a birational morphism
$\pi:M\to M_0$ onto a Kummer orbifold. In particular, there exists a torus 
$A$ and a finite subgroup $F$ of $\Aut(A)$ such that $M_0$ is isomorphic 
to the quotient $A/F.$ Sections \ref{par:Kummer} and \ref{par:VolumeForm}
show that we can write $M_0$ as $A'/F'$ where $F' $ acts on the torus $A'$
by multiplication by $\eta,$ with $\eta\in \{-1, i, e^{2i\pi/3},  -e^{2i\pi/3}\}.$
The unique case which leads to a Calabi-Yau manifold is 
$\eta= e^{2i\pi/3}.$
\end{proof}

\subsection{Infinite center and compact factors}\label{par:GT}$\,$

\vspace{0.16cm}

In this section, our goal is to remove the hypothesis concerning
the center (resp. the compact factors) of $G$ in theorem A and theorem B. 
We first start with the center.

\subsubsection{Infinite center}\label{par:InfiniteCenter}
Let $G$ be a connected semi-simple Lie group without non trivial 
compact factor. Let $Z$ be its center,
$G'=G/Z$ the quotient, 
and $\pi:G\to G'$ the natural projection. Let $\Gamma$ be 
a lattice in $G$ and $\Gamma'=\pi(\Gamma)$  its
projection. The following fact is well known but is hard to localize. 
We present a proof using  Borel density theorem.

\begin{lem}
The group $\Gamma'$ is a lattice in $G'$ and $\Gamma\cap Z$ 
has finite index in~$Z.$ 
\end{lem}

\begin{proof}
Let $1$ denote the neutral element in $G$ and $1'=\pi(1).$
Let $B$ be a neighborhood of $1$ in $G$ such
that $\Gamma\cap B=\{1\}$ and $\pi$ defines a diffeomorphism
from $B$ to its image $B'.$  
Let us assume that $\Gamma'$ is not discrete. In that case, there
exists a sequence  $(\gamma_n')$ of pairwise 
distinct elements of $\Gamma'\cap B' $
which converges toward $1'.$ After extraction of a subsequence, 
we can assume that $\gamma'_n/ \dist(\gamma_n', 1')$ converges
toward an element $v\neq 0$ of the Lie algebra $\g.$   Let
$\gamma_n$ be elements of $\Gamma$ such that $\pi(\gamma_n)=\gamma_n'.$
Let us write $\gamma_n=z_n\epsilon_n$ with $z_n$ in $Z$ and $\epsilon_n$
in $B.$ 
Let $\beta$ be an element of $\Gamma.$ If $n$ is large enough,
$$
[\beta, \gamma_n]=[\beta, \epsilon_n]\in \Gamma\cap B
$$
and therefore $[\beta,\gamma_n]=1.$ 
Thus, $v$ is invariant under the adjoint representation $ad:\Gamma\to \GL(\g).$
Borel density theorem  (see theorem 5.5 of \cite{Raghunathan:book}) implies that $v$ is
invariant under the adjoint representation of $G'.$ Since $G'$ is semi-simple, 
we obtain a contradiction with $v\neq 0,$ proving that $\Gamma'$
is indeed discrete.

Let ${\tilde{\Gamma}}$ be the group $\pi^{-1}(\Gamma').$
Since $\pi$ is a covering, $\tilde{\Gamma}$ is a discrete
subgroup of $G$ containing both $\Gamma$ and $Z.$ 
Since $\Gamma$ is a lattice in $G,$ $\tilde\Gamma$
is also a lattice and $\Gamma$ has finite index in $\tilde \Gamma.$
In particular, $\Gamma\cap Z$ has finite index in $Z.$ 
Moreover, $G'/\Gamma'= G/\tilde\Gamma,$ so that 
$\Gamma'$ is a lattice in $G'.$   
\end{proof}

Let now $\rho:\Gamma\to \GL(V)$ be a finite dimensional
linear representation  of $\Gamma.$ Let $L$ be the Zariski
closure of $\rho(\Gamma),$  $A$
be the center of $L,$ and $\pi_A:L\to L/A$ the natural 
projection.

\begin{lem}
If the image of $\rho$ is discrete, its natural projection 
in $L/A$ is also discrete.
\end{lem}

The proof is along the same lines as the previous one. 
For the sake of simplicity, we slightly change it, 
using a finite generating set $\{\beta_i, 1\leq i\leq l\}$
for $\Gamma.$

\begin{proof}
Let $B$ be a neighborhood of $1$ in $L$ such 
that $\Gamma\cap B=\{1\}.$ Let $U\subset B$ be a 
 neighborhood of $1$ such that $[\rho(\beta_i),U]\subset B$
 for all $1\leq i\leq l.$  Let
 $U'=\pi_A(U).$ Let $\gamma$
 be an element of $\Gamma$ such that $\pi_A(\rho(\gamma))$
 is contained in $U'.$ Let us write $\rho(\gamma)=a\epsilon$
 where $a\in A$ and $\epsilon\in U.$ Since $a$ is in the center
 of $L,$ we see that $[\rho(\beta_i),\rho(\gamma)]$ is contained 
 in $B\cap \Gamma,$ and is thus equal to $1,$ for all $l$
 generators $\beta_i.$ This implies that $\rho(\gamma)$ is in the
 center of $\rho(\Gamma).$ Since $L$ is the Zariski closure
 of $\rho(\Gamma),$ $\rho(\gamma)$ is in the center $A$ of $L.$
As a consequence, $\pi_A(\rho(\Gamma) )$ intersects $U'$
trivially, proving the lemma.
\end{proof}

Let us now assume that the rank of $G$ is at least $2.$
Margulis's arithmeticity  theorem implies that $\Gamma'$
is an arihtmetic lattice in $G'.$ We can then apply 
a result due to Millson, Deligne and 
Raghunathan (see \cite{Margulis:Book}, remark 6.18 (A), page 333),
according to which any subgroup $\Lambda$ of $G$
with $\pi(\Lambda)=\Gamma'$ is a lattice in $G.$ 
Since $[\Gamma',\Gamma']$ has finite index in $\Gamma',$
this implies that $[\Gamma,\Gamma]$ has finite index
in $\Gamma.$ 
 
Let now $\rho:\Gamma\to \Aut(M)$ be a morphism into the
group of automorphisms of a compact K\"ahler manifold
$M.$ Let $\rho^*$ be the morphism given by the action 
on the cohomology $H^*(M,\Z).$ 
Let $V=H^*(M,\Z)\otimes \R.$ The image of $\rho^*$
is a discrete subgroup of $\GL(V).$ With the same 
notations as above, the morphism $\pi_A\circ\rho^*:\Gamma\to L/A$
is trivial on $\Gamma\cap Z$ because $\rho^*(Z\cap \Gamma)$
is contained in $A.$ As a consequence, this morphism
factors through $\Gamma'.$ Since its image is discrete, 
it extends virtually to a morphism of Lie groups $G'\to L/A$ with Zariski
dense image (cf. theorem \ref{thm:SuperMargu}). As a consequence, 
$L/A$ is finite or locally isomorphic to $G,$ because $G$ is semi-simple.

If $L/A$ is finite, the image of $\rho^*$ is virtually abelian, thus 
finite because $[\Gamma,\Gamma]$ has finite index in $\Gamma.$ 
In other words, a finite index subgroup of $\Gamma$ acts trivially
on the cohomology of $M.$ Lieberman-Fujiki's theorem implies
that $\rho(\Gamma)$ is virtually contained in $\Aut(M)^0.$ 

If $L/A$ is locally isomorphic to $G,$ then $L$ is locally isomorphic
to $G\times A$ because $A\to L\to L/A$ is a central extension. Since
$[\Gamma,\Gamma]$ has finite index in $\Gamma,$ $A$ is a finite
group and $\rho^*(Z\cap \Gamma)$ is finite too. In other words, 
$\rho^*$ factors virtually to an almost faithful representation of $\Gamma'.$ 
 Section \ref{chap:HSHRLG}
implies that $G'$ is locally isomorphic to $\SL_3(\R)$ or $\SL_3(\C).$ 
In particular, the center of $G$ is finite, and $\Gamma$ is commensurable
to $\Gamma'.$ 

This proves that theorem A holds even if the center of $G$ is 
infinite. Theorem B follows from theorem A, section \ref{chap:invariant},
and classical algebraic geometry. Section \ref{chap:invariant} needs property (T) for
$\Gamma,$ and this property holds for lattices in  Lie
groups $G$ as soon as $\rk_\R(\g)\geq 2$ and $\g$ is simple 
(the center of $G$ can be infinite, see \cite{BHV}). As a consequence, theorem B can
be generalized as follows.

\vspace{0.16cm}

\begin{thm-BB}
Let $G$ be a connected  real  Lie group with a simple, higher rank Lie algebra $\g.$ 
Let $\Gamma$ be a lattice in $G.$ 
 Let $M$ be a connected compact K\"ahler
manifold of dimension $3.$ If there is a morphism $\rho:\Gamma\to \Aut(M)$ 
with infinite image,  then $M$ has a birational morphism onto
a Kummer orbifold, or $M$ is isomorphic to one of the following
\begin{enumerate}
\item a projective bundle $\P(E)$ for some rank $2$ vector bundle $E\to \P^2(\C),$ 
\item a principal torus bundle over $\P^2(\C),$ 
\item a product $\P^2(\C)\times B$ of the plane by a curve of genus $g(B)\geq 2,$ 
\item the projective space $\P^3(\C).$ 
\end{enumerate}
In all cases, $\g$ is isomorphic to $\sll_n({\mathbf{K}})$ with $n=3$ or $4$ 
and ${\mathbf{K}}=\R$ or $\C.$
\end{thm-BB} 

\subsubsection{Compact factors}\label{par:CompactFactors}

Let now $G$ be a higher rank semi-simple Lie group. Let
$K$ be a maximal, connected, normal and compact subgroup of $G.$
Let ${\overline{G}}$ be the quotient $G/K$ and $\pi:G\to {\overline{G}}$ be the 
natural projection. Let $\Gamma$ be a lattice in $G.$ 
Since $\Gamma$ is discrete, $\Gamma\cap K$ is finite
and ${\overline{\Gamma}}=\pi(\Gamma)$ is a lattice in~${\overline{G}}.$ 
Let $\rho:\Gamma\to \Aut(M)$ be a morphism from $\Gamma$
to the group of automorphisms of a 
connected compact K\"ahler manifold $M.$
Let $\rho^*:\Gamma\to \GL(H^*(M,\Z))$ be the action on the
cohomology of $M.$ From Selberg's lemma,
there is  a finite index subgroup $\Gamma_1$ of $\Gamma$
such that $\rho^*(\Gamma_1)$ is torsion free. Changing
$\Gamma$ into $\Gamma_1,$ we have $\rho^*(K\cap \Gamma)=\{1\}.$
In other words, $\rho^*$ factors through ${\overline{\Gamma}}.$
Sections \ref{chap:HSHRLG} and \ref{par:InfiniteCenter} imply that
\begin{itemize}
\item the image of $\rho^*$ is finite, or
\item ${\overline{G}}$ is locally isomorphic to $\SL_3(\R)$ or $\SL_3(\C)$
and the action of ${\overline{\Gamma}}$ on $H^*(M,\R)$ extends virtually
to a non trivial representation of ${\overline{G}}.$
\end{itemize}
In the first case, the image of $\rho$ is virtually contained in $\Aut(M)^0.$
In the second case, there is a section $s$ of $\pi:\Gamma \to {\overline{\Gamma}}$
over a finite index subgroup ${\overline{\Gamma}}_1$ of ${\overline{\Gamma}}.$ Changing 
$\Gamma$ into $s({\overline{\Gamma}}_1),$ we can then apply theorem A. 
The final form of theorem A  can now be stated as follows. 

\vspace{0.16cm}

\begin{thm-AA}
Let $G$ be a connected semi-simple real Lie group.
Let $K$ be the maximal compact, connected, and normal subgroup of $G.$ 
Let $\Gamma$ be an irreducible lattice in $G.$
Let $M$
be a connected compact K\"ahler manifold of dimension $3,$ and $\rho:\Gamma\to \Aut(M)$ be 
a morphism. If the real rank of $G$ is at least $2,$ then one of the following holds
\begin{itemize}
\item the image of $\rho$ is virtually contained in the connected component of
the identity $\Aut(M)^0,$ or
\item the morphism $\rho$ is virtually a Kummer example. 
\end{itemize}
In the second case, $G/K$ is locally isomorphic
to $\SL_3(\R)$ or $\SL_3(\C)$ and $\Gamma$ is commensurable to 
$\SL_3(\Z)$ or $\SL_3({\mathcal{O}}_{d}),$ where ${\mathcal{O}}_{d}$
is the ring of integers in an imaginary quadratic number field $\Q({\sqrt{d}})$ 
for some negative integer $d.$
\end{thm-AA}

%
%

 

\begin{thebibliography}{10}

\bibitem{Auslander-Green-Hahn:book}
L.~Auslander, L.~Green, and F.~Hahn.
\newblock {\em Flows on homogeneous spaces}.
\newblock With the assistance of L. Markus and W. Massey, and an appendix by L.
  Greenberg. Annals of Mathematics Studies, No. 53. Princeton University Press,
  Princeton, N.J., 1963.

\bibitem{BPV:book}
Barth, Peters, and van~de Ven.
\newblock {\em Compact complex surfaces}.
\newblock Springer-Verlag, 1984.

\bibitem{BHV}
Bachir Bekka, Pierre de~la Harpe, and Alain Valette.
\newblock {\em Kazhdan's property ({T})}, volume~11 of {\em New Mathematical
  Monographs}.
\newblock Cambridge University Press, Cambridge, 2008.

\bibitem{Benoist:cours}
Y.~Benoist.
\newblock Sous-groupes discrets des groupes de {L}ie.
\newblock {\em European Summer School in Group Theory}, pages 1--72, 1997.

\bibitem{Benoist-Labourie:1993}
Yves Benoist and Fran{\c{c}}ois Labourie.
\newblock Sur les diff\'eomorphismes d'{A}nosov affines \`a\ feuilletages
  stable et instable diff\'erentiables.
\newblock {\em Invent. Math.}, 111(2):285--308, 1993.

\bibitem{Benveniste-Fisher:2005}
E.~Jerome Benveniste and David Fisher.
\newblock Nonexistence of invariant rigid structures and invariant almost rigid
  structures.
\newblock {\em Comm. Anal. Geom.}, 13(1):89--111, 2005.

\bibitem{Bochner-Montgomery:1946}
S.~Bochner and D.~Montgomery.
\newblock Locally compact groups of differentiable transformations.
\newblock {\em Ann. of Math. (2)}, 47:639--653, 1946.

\bibitem{Cairns-Ghys:1997}
G.~Cairns and {\'E}.~Ghys.
\newblock The local linearization problem for smooth {${\rm SL}(n)$}-actions.
\newblock {\em Enseign. Math. (2)}, 43(1-2):133--171, 1997.

\bibitem{Campana-Peternell:survey}
F.~Campana and Th. Peternell.
\newblock Cycle spaces.
\newblock In {\em Several complex variables, {VII}}, volume~74 of {\em
  Encyclopaedia Math. Sci.}, pages 319--349. Springer, Berlin, 1994.

\bibitem{Campana:2004}
Fr{\'e}d{\'e}ric Campana.
\newblock Orbifoldes \`a premi\`ere classe de {C}hern nulle.
\newblock In {\em The {F}ano {C}onference}, pages 339--351. Univ. Torino,
  Turin, 2004.

\bibitem{Cantat:ENS}
Serge Cantat.
\newblock Version k\"ahl\'erienne d'une conjecture de {R}obert {J}. {Z}immer.
\newblock {\em Ann. Sci. \'Ecole Norm. Sup. (4)}, 37(5):759--768, 2004.

\bibitem{Cantat:Cremona}
Serge Cantat.
\newblock Groupes de transformations birationnelles du plan.
\newblock {\em preprint}, pages 1--50, 2007.

\bibitem{Corlette:1992}
Kevin Corlette.
\newblock Archimedean superrigidity and hyperbolic geometry.
\newblock {\em Ann. of Math. (2)}, 135(1):165--182, 1992.

\bibitem{delaHarpe-Valette}
P.~de~la Harpe and A.~Valette.
\newblock La propri\'et\'e {$(T)$} de {K}azhdan pour les groupes localement
  compacts (avec un appendice de {M}arc {B}urger).
\newblock {\em Ast\'erisque}, (175):158, 1989.
\newblock With an appendix by M. Burger.

\bibitem{Debarre:book}
Olivier Debarre.
\newblock {\em Tores et vari\'et\'es ab\'eliennes complexes}, volume~6 of {\em
  Cours Sp\'ecialis\'es [Specialized Courses]}.
\newblock Soci\'et\'e Math\'ematique de France, Paris, 1999.

\bibitem{Demailly:LNM}
Jean-Pierre Demailly.
\newblock {$L^2$} vanishing theorems for positive line bundles and adjunction
  theory.
\newblock In {\em Transcendental methods in algebraic geometry ({C}etraro,
  1994)}, volume 1646 of {\em Lecture Notes in Math.}, pages 1--97. Springer,
  Berlin, 1996.

\bibitem{Demailly-Paun:2004}
Jean-Pierre Demailly and Mihai Paun.
\newblock Numerical characterization of the {K}\"ahler cone of a compact
  {K}\"ahler manifold.
\newblock {\em Ann. of Math. (2)}, 159(3):1247--1274, 2004.

\bibitem{Dinh-Nguyen:2009}
Tien-Cuong Dinh and Vi{\^e}t-Anh Nguyen.
\newblock Comparison of dynamical degrees for semi-conjugate meromorphic maps.
\newblock pages 1--23. 2009.

\bibitem{Farb-Masur:1998}
Benson Farb and Howard Masur.
\newblock Superrigidity and mapping class groups.
\newblock {\em Topology}, 37(6):1169--1176, 1998.

\bibitem{Farb-Shalen:1999}
Benson Farb and Peter Shalen.
\newblock Real-analytic actions of lattices.
\newblock {\em Invent. Math.}, 135(2):273--296, 1999.

\bibitem{Farb-Shalen:2000}
Benson Farb and Peter Shalen.
\newblock Lattice actions, 3-manifolds and homology.
\newblock {\em Topology}, 39(3):573--587, 2000.

\bibitem{Farb-Shalen:2002}
Benson Farb and Peter~B. Shalen.
\newblock Real-analytic, volume-preserving actions of lattices on 4-manifolds.
\newblock {\em C. R. Math. Acad. Sci. Paris}, 334(11):1011--1014, 2002.

\bibitem{Fisher:preprint}
David Fisher.
\newblock Groups acting on manifolds: Around the zimmer program.
\newblock In {\em preprint}, pages 1--84.

\bibitem{Fulton-Harris:book}
William Fulton and Joe Harris.
\newblock {\em Representation theory}, volume 129 of {\em Graduate Texts in
  Mathematics}.
\newblock Springer-Verlag, New York, 1991.
\newblock A first course, Readings in Mathematics.

\bibitem{Ghys:1993}
{\'E}tienne Ghys.
\newblock Sur les groupes engendr\'es par des diff\'eomorphismes proches de
  l'identit\'e.
\newblock {\em Bol. Soc. Brasil. Mat. (N.S.)}, 24(2):137--178, 1993.

\bibitem{Gromov:Enseignement}
Mikha{\"{\i}}l Gromov.
\newblock On the entropy of holomorphic maps.
\newblock {\em Enseign. Math. (2)}, 49(3-4):217--235, 2003.

\bibitem{Guedj:ETDS}
Vincent Guedj.
\newblock Entropie topologique des applications m\'eromorphes.
\newblock {\em Ergodic Theory Dynam. Systems}, 25(6):1847--1855, 2005.

\bibitem{Guedj:survey}
Vincent Guedj.
\newblock Propri{\'e}t{\'e}s ergodiques des applications rationnelles.
\newblock {\em Survey}, pages 1--130, 2007.

\bibitem{Guentner-all:2005}
Erik Guentner, Nigel Higson, and Shmuel Weinberger.
\newblock The {N}ovikov conjecture for linear groups.
\newblock {\em Publ. Math. Inst. Hautes \'Etudes Sci.}, (101):243--268, 2005.

\bibitem{Hartshorne:LNM}
Robin Hartshorne.
\newblock {\em Ample subvarieties of algebraic varieties}.
\newblock Notes written in collaboration with C. Musili. Lecture Notes in
  Mathematics, Vol. 156. Springer-Verlag, Berlin, 1970.

\bibitem{Katok-Lewis:1991}
A.~Katok and J.~Lewis.
\newblock Local rigidity for certain groups of toral automorphisms.
\newblock {\em Israel J. Math.}, 75(2-3):203--241, 1991.

\bibitem{Katok-Lewis:1996}
A.~Katok and J.~Lewis.
\newblock Global rigidity results for lattice actions on tori and new examples
  of volume-preserving actions.
\newblock {\em Israel J. Math.}, 93:253--280, 1996.

\bibitem{Knapp:book}
Anthony~W. Knapp.
\newblock {\em Lie groups beyond an introduction}, volume 140 of {\em Progress
  in Mathematics}.
\newblock Birkh\"auser Boston Inc., Boston, MA, second edition, 2002.

\bibitem{Kobayashi:book}
Shoshichi Kobayashi.
\newblock {\em Differential geometry of complex vector bundles}, volume~15 of
  {\em Publications of the Mathematical Society of Japan}.
\newblock Princeton University Press, Princeton, NJ, 1987.
\newblock Kan{\^o} Memorial Lectures, 5.

\bibitem{Kodaira-Spencer:1959}
K.~Kodaira and D.~C. Spencer.
\newblock A theorem of completeness of characteristic systems of complete
  continuous systems.
\newblock {\em Amer. J. Math.}, 81:477--500, 1959.

\bibitem{Lieberman:1978}
D.~I. Lieberman.
\newblock Compactness of the {C}how scheme: applications to automorphisms and
  deformations of {K}\"ahler manifolds.
\newblock In {\em Fonctions de plusieurs variables complexes, III (S\'em.
  Fran\c cois Norguet, 1975--1977)}, pages 140--186. Springer, Berlin, 1978.

\bibitem{Margulis:Book}
G.~A. Margulis.
\newblock {\em Discrete subgroups of semisimple {L}ie groups}, volume~17 of
  {\em Ergebnisse der Mathematik und ihrer Grenzgebiete (3) [Results in
  Mathematics and Related Areas (3)]}.
\newblock Springer-Verlag, Berlin, 1991.

\bibitem{Witte:prebook}
Dave~Witte Morris.
\newblock {\em Introduction to arithmetic groups}.
\newblock Preprint. 2003.

\bibitem{Mostow:book}
G.~D. Mostow.
\newblock {\em Strong rigidity of locally symmetric spaces}.
\newblock Princeton University Press, Princeton, N.J., 1973.
\newblock Annals of Mathematics Studies, No. 78.

\bibitem{Oguiso-Sakurai:2001}
Keiji Oguiso and Jun Sakurai.
\newblock Calabi-{Y}au threefolds of quotient type.
\newblock {\em Asian J. Math.}, 5(1):43--77, 2001.

\bibitem{Raghunathan:book}
M.~S. Raghunathan.
\newblock {\em Discrete subgroups of {L}ie groups}.
\newblock Springer-Verlag, New York, 1972.
\newblock Ergebnisse der Mathematik und ihrer Grenzgebiete, Band 68.

\bibitem{Tits:1962}
J.~Tits.
\newblock Espaces homog\`enes complexes compacts.
\newblock {\em Comment. Math. Helv.}, 37:111--120, 1962/1963.

\bibitem{VGS:EMS}
E.~B. Vinberg, V.~V. Gorbatsevich, and O.~V. Shvartsman.
\newblock Discrete subgroups of {L}ie groups [{MR} 90c:22036].
\newblock In {\em Lie groups and Lie algebras, II}, volume~21 of {\em
  Encyclopaedia Math. Sci.}, pages 1--123, 217--223. Springer, Berlin, 2000.

\bibitem{Voisin:book}
Claire Voisin.
\newblock {\em Th\'eorie de {H}odge et g\'eom\'etrie alg\'ebrique complexe},
  volume~10 of {\em Cours Sp\'ecialis\'es [Specialized Courses]}.
\newblock Soci\'et\'e Math\'ematique de France, Paris, 2002.

\bibitem{Winkelmann:LNM}
J{\"o}rg Winkelmann.
\newblock {\em The classification of three-dimensional homogeneous complex
  manifolds}, volume 1602 of {\em Lecture Notes in Mathematics}.
\newblock Springer-Verlag, Berlin, 1995.

\bibitem{Wolf:book}
Joseph~A. Wolf.
\newblock {\em Spaces of constant curvature}.
\newblock Publish or Perish Inc., Houston, TX, fifth edition, 1984.

\bibitem{Zaffran-Wang:2009}
D.~Zaffran and Z.~Z. Wang.
\newblock A remark for the hard lefschetz theorem for k\"ahler orbifolds.
\newblock {\em Proc. Amer. Math. Soc.}, 137:2497--2501, 2009.

\bibitem{Zimmer:1984}
R.~J. Zimmer.
\newblock Kazhdan groups acting on compact manifolds.
\newblock {\em Invent. Math.}, 75(3):425--436, 1984.

\end{thebibliography}

\end{document}